\documentclass[reqno,twoside,11pt]{amsart}
\usepackage{amsbsy,amssymb,amscd,amsfonts,latexsym,amstext,delarray,amsmath,graphicx,color,caption,blkarray,mathtools}
\usepackage{mathrsfs}
\usepackage{tikz}
\usepackage{graphicx}
\usepackage{hyperref}
\input xy

\usetikzlibrary{cd}
\usetikzlibrary{matrix,arrows}

\xyoption{all}
\usepackage{euscript}
\usepackage{multirow,charter}
\usepackage{etex, pictexwd,dcpic}
\usepackage[normalem]{ulem}

\DeclareMathOperator{\M}{\mathbb M}

    \advance \topmargin by -\headheight
    \advance \topmargin by -\headsep

    \evensidemargin \oddsidemargin
\newtheorem {theorem}    {Theorem}[section]

\newtheorem {lemma}      [theorem]    {Lemma}
\newtheorem {corollary}  [theorem]    {Corollary}
\newtheorem {proposition}[theorem]    {Proposition}

\theoremstyle{definition}
\newtheorem*{definition}{Definition}
\theoremstyle{remark}

    \advance \topmargin by -\headheight
    \advance \topmargin by -\headsep

    \evensidemargin \oddsidemargin
\DeclareMathOperator{\SL}{SL}
\DeclareMathOperator{\GL}{GL}

\DeclareMathOperator{\Aut}{Aut}

\DeclareMathOperator{\Ad}{Ad}
\DeclareMathOperator{\ad}{ad}

\DeclareMathOperator{\re}{{re}}

\def\wgal{\widetilde w}
\renewcommand{\a}{\alpha}
\renewcommand{\b}{\beta}

\newcommand{\K}{{\mathbb K}}
\newcommand{\Z}{{\mathbb Z}}
\newcommand{\R}{{\mathbb R}}
\newcommand{\C}{{\mathbb C}}
\newcommand{\N}{{\mathbb N}}
\newcommand{\Q}{{\mathbb Q}}
\author{Lisa Carbone}
\email{lisa.carbone@rutgers.edu}
\author{Elizabeth Jurisich}
\email{jurisiche@cofc.edu}

\title{A MAGNUS GROUP CONSTRUCTION FOR A CLASS OF BORCHERDS ALGEBRAS}

\begin{document}

\begin{abstract}
    We construct a group associated to a class of Borcherds algebras that 
    admit a direct sum decomposition into a  Kac--Moody (or semi-simple) subalgebra and a pair of free Lie subalgebras. 
    Such Borcherds algebras have no mutually orthogonal  imaginary simple roots.
    Our group is a semi-direct product of a Kac--Moody (or semi-simple) group and a Magnus group of invertible formal power series corresponding to a basis of a certain highest weight module determined by the simple imaginary roots. We show that our group is independent of this choice of basis, up to isomorphism. We apply our construction to a number of concrete examples, such as certain Borcherds algebras formed using root lattices of hyperbolic Kac--Moody algebras, the Monster Lie algebra,  Monstrous Lie algebras of Fricke type and the gnome Lie algebra.
\end{abstract}

\maketitle

\section{Introduction}

Borcherds algebras are infinite dimensional Lie algebras which are a generalization of Kac--Moody algebras originally studied by Richard Borcherds in the context of Monstrous Moonshine \cite{BoInvent}. 

One obstruction to associating a Kac--Moody-type group to a Borcherds algebra $\mathfrak g$ is the appearance of  imaginary simple roots, whose root vectors  may not act locally nilpotently on a given standard irreducible module of $\mathfrak g$ \cite{JURISICH2004149}.

For the class of Borcherds algebras that admit a direct sum decomposition into a  Kac--Moody (or semi-simple) subalgebra $\mathfrak g_J$ and  free Lie subalgebras $\mathfrak u^\pm$ (as in Theorem 5.1 of \cite{JurJPAA},  see also Theorem~\ref{JurThm}), here we construct an associated Lie group analog which  is a semi-direct product of a Kac--Moody (or semi-simple) group and a Magnus group. The latter is a group of  formal power series. A key property of this class of Borcherds algebras is that they have  no mutually orthogonal  imaginary simple roots \cite{JurJPAA}.

 The construction presented in this work provides a solution to the fundamental problem of associating a group-theoretic analog to Borcherds algebras.   Since Borcherds algebras contain imaginary simple roots, and imaginary root vectors do not act locally nilpotently on standard modules of interest, the standard group constructions fail (see \cite{CJM} for more exposition and detail).

Let $I$ be an index set for all simple roots and let $J\subset I$ be an index set for the real simple roots. Associated to the decomposition of the Borcherds algebra $$\mathfrak g = \mathfrak u^+ \oplus (\mathfrak g_J + \mathfrak h) \oplus \mathfrak u^-$$ of \cite{JurJPAA}, we construct a group $G$, which is the semi-direct product $$G=   G(S^\prime) \rtimes G_J. $$ 
The set $S^\prime$ is a  choice of basis for a space $V^\prime$, defined as  $V^\prime=\coprod_{i \in I\backslash J}{ U}(\mathfrak n^-_J)\cdot f_i$, where $\mathfrak  u^-$ is the free subalgebra $L(V')$, $\mathfrak{n}^-_J$ is the nilpotent subalgebra corresponding to the negative roots and $f_i$ are Chevalley generators corresponding to the imaginary simple roots. The group $G_J$ is a Kac--Moody (or semi-simple) group associated to $\mathfrak g_J$, relative to a representation of $\mathfrak g_J$ on the tensor algebra $W= T(V^\prime)$. The group  $G(S^\prime)= \exp(\widehat{L}(S^\prime))$ is a group of invertible formal power series in non-commuting variables from $S^\prime$ where $\widehat{L}(S^\prime)$  is a completion of the free subalgebra $\mathfrak u^-$.  Our construction of the group $G$ is invariant, up to isomorphism, under the choice of basis of $S^\prime$ for $V^\prime$ (see Theorem~\ref{independent}).



To construct the group $G$, we reinterpret the $\mathfrak g_J$-module $T(V')$ as a free associative algebra $A(S^\prime)$ on the basis $S^\prime$ of $V'$. In this way,  we transfer the module action to an action by derivations on formal polynomials.

In Subsection~\ref{Magnus}, we  recall the Magnus group construction associated with a free Lie algebra, then apply the construction to free subalgebra $\mathfrak u^-$ of $\mathfrak g$. The Magnus group turns out to be the natural group-theoretic object for encoding the free Lie subalgebra $\mathfrak{u}^-$. As the group of invertible formal power series in non-commuting variables, the Magnus group $G(S')$ provides a formal completion that ensures the exponential map is well-defined even in the absence of local nilpotency.

In Subsection~\ref{KMgroup},  we recall the construction of an adjoint Kac--Moody group relative to  a representation of $\mathfrak g_J$ on the tensor algebra $W= T(V^\prime)$. In Subsection~\ref{Polynomials}, we  define an action of $\mathfrak{g}_J$ on the elements of $T(V^\prime)$, viewed  as formal polynomials. 
The action of $\mathfrak g_J$ on $\mathfrak u^-$ as a $\mathfrak g_J$-module induces an action of the group $G_J$ on the  group $G(S^\prime)$.

In Subsection~\ref{SDproduct} we put these ingredients together to assemble the semi-direct product $G=   G(S^\prime) \rtimes G_J. $ 
 We thus obtain a group that mirrors the  decomposition $\mathfrak{g} \cong \mathfrak{u}^- \oplus (\mathfrak{g}_J + \mathfrak{h}) \oplus \mathfrak{u}^+$ (though we only use the negative imaginary roots for our group construction).

 Finally, in Section~\ref{examples}, 
we apply our construction to the following concrete examples:
\begin{itemize}
\item Certain Borcherds algebras formed using root lattices of hyperbolic Kac--Moody algebras:\\
    - A Borcherds algebra associated with the hyperbolic lattice \( H(3) \). \newline
    - A Borcherds algebra associated with the hyperbolic  lattice $E_{10}$. 
    \item The Monster Lie algebra, whose structure plays a key role in Monstrous Moonshine.
      \item Monstrous Lie algebras of Fricke type as in \cite{CarDuke}.
    \item The gnome Lie algebra of \cite{BGN98}.

\end{itemize}

 In comparison with minimal (incomplete) and maximal (complete) Kac--Moody groups, our semi-direct product $G = G(S') \rtimes G_J$ is a `hybrid'. The component $G(S^\prime)= \exp(\widehat{L}(S^\prime)$ corresponding to the imaginary roots is associated to the completion $\widehat{L}(S^\prime))$ and may be viewed as a complete group, while the component $G_J$ corresponding to the real roots is incomplete.  If a fully complete group is needed, one may take a further completion of the Kac--Moody group $G_J$.

 Our construction raises many new questions. For example, does $G = G(S') \rtimes G_J$ carry a natural topology or pro-algebraic structure induced by the completion? We hope to consider this question in future work.

\section{Preliminary Material}


We denote the non-negative integers by $\mathbb N$, and $\mathbb K\subset \mathbb C$ a  subfield of $\mathbb C$ such as $\Q$, $\R$ or $\C$.




  
  \subsection{Universal Enveloping Algebras}

Let $L$ be a Lie algebra over $\mathbb K$.  Let $U(L)$ denote the universal enveloping algebra of $L$. 

Let g be a Lie algebra and U(g) be its universal enveloping algebra.
Let ${s_1,s_2,…,s_n } $ be an ordered basis for the Lie algebra $\mathfrak g$. We will use a Poincaré-Birkhoff-Witt basis of $U(\mathfrak g)$ denoted
\begin{equation}
s_{i_1}^{a_1} s_{i_2}^{a_2} \cdots s_{i_k}^{a_k}
\end{equation}
where $$1 \leq i_1<i_2
 <\ldots <i_k
 \leq n $$ and $a_j \geq 0 $ are integers.



Given a countable set $X$, let $L(X)$ denote the free Lie algebra generated by $X$.  For a vector space $V$ with basis $X$ we also denote $L(X)$ by $L(V)$. The universal enveloping algebra of $L(X)$ is the free associative algebra $A(X)$, which is isomorphic to the tensor algebra $T(V)$ for $V= \text{Span}
\{X\}$. Recall that 
$$T(V) = \bigoplus_{n=0}^{\infty} T^n(V)$$
where $T^0(V) = \mathbb K$, $T^1(V) = V$, $T^n(V) = V^{\otimes n}$ for $n \ge 1$.

The completion of $T(V)$, denoted $\widehat{T}(V)$, is the algebra of formal power series in non-commuting variables corresponding to  basis elements of $V$:
\[ \widehat{T}(V) = \prod_{n=0}^{\infty} T^n(V). \]
Elements of $\widehat{T}(V)$ are formal infinite sums of the form
$ \sum_{n=0}^{\infty} v_n$ where $v_n \in T^n(V)$. This determines an $\N$-grading where $(uv)_n = \sum_{m=0}^{n} u_m v_{n-m}$.

Since $\widehat{T}(V)$ is identified with the direct product $\prod_{n=0}^{\infty} T^n(V)$ as vector spaces, the topology coincides with the {product topology}, where each component  $T^n(V)$ is equipped with the {discrete topology}.

\subsection{Borcherds algebras}

Borcherds algebras, also called  generalized Kac--Moody algebras, can be defined using a generalization of a Cartan matrix $A$ \cite{B88}.  We assume here that the matrix $A$ is symmetric, but the results hold for symmetrizable matrices. 

Let $I$ be a  countable index set and let $A = (a_{ij})_{i,j \in
I}$ be a 
matrix with entries in ${\mathbb K}$, satisfying the following conditions:

\begin{description}
\item[(B1)]  $A$ is symmetric.
\item[(B2)]  If $ i\neq j$ ($i,j \in I$), then $a_{ij}~\leq~0 $.
\item[(B3)]  If $a_{ii} > 0$ ($i \in I$), then $\frac{2a_{ij} }{ a_{ii}}
\in \mathbb Z $ for all $j  \in I$. 
\end{description}

Let $\mathfrak g$ be the Lie algebra with generators
$h_{i}, e_i, f_i$, $i \in I$, and the following defining
relations: for all $i,j,k \in I$,

\begin{description} 
 \item[(R1)] $\left[ h_{i}, h_{j}\right] =0$,
 \item[(R2)] $\left[ h_{i}, e_k\right]   -   a_{jk} e_k =0$,
 \item[(R3)] $\left[ h_{i}, f_k\right]  +   a_{jk} f_k =0$,
 \item[(R4)] $\left[e_i , f_j \right] -   \delta_{ij} h_{i} =0$,
 \item[(R5)] $ (\text{ad} e_i)^{\frac{-2a_{ij}} {a_{ii}} + 1}e_j =0$ and 
        $(\text{ad} f_i)^{\frac{-2a_{ij} }{a_{ii}} + 1}f_j =0 
\mbox{ for all } i \neq j \mbox{ with } a_{ii} > 0$, 
 \item[(R6)] $[e_i, e_j]=0$ and $ [f_i,f_j]=0$ whenever $a_{ij} =0$.
\end{description}

   \begin{definition} 
 The Lie algebra $\mathfrak g$ over $\mathbb K$ is {\em the Borcherds
 algebra} associated to the matrix A. 
\end{definition}

We also call the Lie algebra $\mathfrak g/\mathfrak{c}$, where $\mathfrak{c}$ is a central ideal, a Borcherds algebra. 
The Cartan subalgebra $\mathfrak h$ is the abelian subalgebra spanned by the $\{h_i\}_{i\in I} $. 

Since the defining matrix $A$ may not have full rank we extend $\mathfrak{h}$ by an abelian algebra of degree derivations $\mathfrak D$ to $\mathfrak h^e = \mathfrak D \ltimes \mathfrak h $,  and form $\mathfrak g^e$. We extend the Cartan subalgebra by sufficient degree derivations to ensure that the simple roots defined below are linearly independent. We can then `specialize' the roots if desired, by restricting them to $\mathfrak h$ or some smaller subspace. 

Borcherds algebras have the following properties (for details see \cite{B88},\cite{BoInvent}, and \cite{Jur96}). Define simple roots $\alpha_i(h_j) = a_{ij}$ for $h_j \in (\mathfrak h^ e)^*$. Let $\Delta$ be the set of roots and denote the positive (respectively negative) roots by  $\Delta_+$
(respectively $ \Delta_-$).
    The {\it root lattice} $Q$ is the $\Z$-span of the simple roots $\alpha_i$:
          \[
          Q = \bigoplus_{i\in I } \mathbb{Z}\alpha_i.
          \]
         
The Borcherds algebra $\mathfrak{g}$ has triangular decomposition
          \[
          \mathfrak{g} = \mathfrak{n}^- \oplus \mathfrak{h} \oplus \mathfrak{n}^+
          \]
          where $\mathfrak{h}$ is the Cartan subalgebra, and $\mathfrak{n}^-$ and $\mathfrak{n}^+$ are the direct sums of negative and positive root spaces, respectively:

          \begin{equation}
\mathfrak g = \coprod_{\varphi \in \Delta_+} \mathfrak g_{\varphi} \oplus 
\mathfrak h \oplus \coprod_{\varphi \in \Delta_-}\mathfrak g_{\varphi}.
\label{eq:rts}\end{equation}
The spaces $\mathfrak g_{\varphi}$ are finite-dimensional. We denote the set of {\em real} roots by $\Delta_{\re}$. The set of {\em imaginary} roots is $\Delta \backslash \Delta_{\re}$. The imaginary simple
roots are those $\alpha_i$ for which $a_{ii} \leq 0$.

\subsection{Symmetric bilinear form} 
There exists a symmetric bilinear form $(\cdot,\cdot)$ on Q which extends to $(\mathfrak{h^e})^*$  and $\mathfrak g^e$ satisfying:
          \begin{itemize}
              \item $(\alpha_i, \alpha_j) = a_{ij}$, where $a_{ij}$ are the entries of the  Borcherds Cartan matrix.
              
              \item For all
$\varphi, \psi \in \Delta$ such that $\varphi + \psi \neq 0$,
$$(\mathfrak g_\varphi, \mathfrak g_\psi) = (\mathfrak h^e, \mathfrak g_\varphi) =0.$$
          \end{itemize}
          
          The symmetric bilinear form on the span of the $\alpha_i$, $i \in I$,
is defined on $\mathfrak h$ by $$(h_i,h_j)_{\mathfrak h} =
(\alpha_i, \alpha_j) = a_{ij}.$$ This form can be extended by fixing a choice of form on $\mathfrak d $  satisfying the conditions $(d,h_j)= \alpha_j(d)$ for
$d \in \mathfrak d$ and $j \in I$.  This form can be constructed on $\mathfrak h^e$ to be consistent with the form on $(\mathfrak h^e)^*$. There is a unique invariant
symmetric bilinear form $(\cdot, \cdot)_{\mathfrak g^e}$ on $\mathfrak g^e$
that extends
$(\cdot, \cdot)_{\mathfrak h^e}$. This form satisfies the condition
$$(\mathfrak g_\varphi, \mathfrak g_\psi)_{\mathfrak g^e} = 
(\mathfrak h^e,g_\varphi)_{\mathfrak g^e} =0$$ for all $\varphi,
\psi \in \Delta $ such that $\varphi + \psi \neq 0$.
Given any $\varphi = \sum_{i
\in I} n_i \alpha_i \in \Delta$ for $n_i \in \mathbb Z$, we define
$h_{\varphi} = \sum_{i\in I}n_i h_i$. Then for $a \in \mathfrak g_\varphi$,
$b \in \mathfrak g_{-\varphi}$, the  
form $(\cdot,\cdot)_{\mathfrak g^e}$ satisfies the condition
$$[a,b]=(a,b)_{\mathfrak g^e}h_{\varphi}.$$ 
In particular, $(e_i,f_j)_{\mathfrak g^e}= \delta_{ij}$ for $i,j \in I$ (see
\cite{Jur96} for details,  see also \cite{Kac90}, \cite{MP95}).


   Let $\mathfrak{g}$ be a Borcherds algebra in which no distinct imaginary simple roots are pairwise orthogonal. Equivalently, the defining matrix satisfies $a_{ij}\neq 0$ for all $i$ such $a_{ii}\leq 0$. 
 Let $J \subset I$ be the set $\{ i \in I
\mid \alpha_i \in \Delta_{\re} \} = \{ i \in I \mid a_{ii} >0\}$. 


Note that the
matrix $(a_{ij})_{i,j \in J}$ is a generalized Cartan matrix.  Let
$\mathfrak g_J$ be the Kac--Moody or semi-simple Lie algebra associated to this matrix.
Then $\mathfrak g_J = \mathfrak n^+_J \oplus \mathfrak h_J \oplus \mathfrak n_J^-$, and
$\mathfrak g_J$ is isomorphic to the subalgebra of $\mathfrak g$ generated by $\{e_i, f_i\}$ with $i \in J$. We let $\Delta_J$ denote the set of roots of $\mathfrak g_J$ and $\Delta_J^{\re}$ the subset of real roots.

\subsection{Modules} For a module $X$ over $\mathfrak{g}$, and a linear functional $\lambda \in (\mathfrak{h}^e)^*$, the $\lambda$-weight space $X_\lambda$ is defined as:
    \[ X_\lambda = \{v \in X \mid h.v = \lambda(h)v \text{ for all } h \in \mathfrak{h}^e\}. \]
   A module $X$ is called a {\it weight module} if it can be decomposed into a direct sum of its weight spaces: $X = \bigoplus_{\lambda \in (\mathfrak{h}^e)^*} X_\lambda$, and each weight space $X_\lambda$ has finite dimension.

 A $\mathfrak{g}$-module $X$ is a {\it highest weight module} with {highest weight $\Lambda \in (\mathfrak{h}^e)^*$} if there exists a non-zero vector $v_\Lambda \in X_\Lambda$ (called the {highest weight vector}) such that $\mathfrak{n}^+.v_\Lambda = 0$, where $\mathfrak{n}^+$ is the nilpotent subalgebra corresponding to positive roots, that is, $e_i.v_\Lambda = 0$ for all simple root generators $e_i$.
  The module $X$ is generated by $v_\Lambda$ as a $\mathfrak{g}$-module.

 All weights $\mu$ of a highest weight module $X$ are of the form $\Lambda - \sum_{i\in I} k_i \alpha_i$, where $k_i \in \mathbb{N}$ and $\alpha_i$ are the simple roots of $\mathfrak{g}$.

A highest weight module $X$ with highest weight $\Lambda$ is called a {standard module} if $\Lambda$ satisfies the following conditions related to the structure of $\mathfrak{g}$: $\Lambda$ is a dominant integral weight, meaning $\Lambda(h_i) \in \mathbb{N}$  for all $i \in I$ and $\Lambda(h_i) \geq 0$ for all $i \in J$.
The highest weight vector $v_\Lambda$ satisfies:  $f_i.v_\Lambda = 0$ for all $i \in I$ such that $\Lambda(h_i) = 0$ and $f_i^{\Lambda(h_i)+1}.v_\Lambda = 0$ for all $i \in J$.

 
Given a dominant integral weight $\Lambda$ there is a unique (up to isomorphism) 
standard irreducible highest weight $\mathfrak g^e$-module with highest weight $\Lambda \in P_+$,
denoted 
$L(\Lambda)$ \cite{JurJPAA}. Note that a standard module $X$ can have infinitely many
(independent) elements of the form $f_i^k \cdot x$ where $k \in \mathbb N$, $i \in
I\backslash I_0$ and
$\Lambda (h_i) \neq 0$. 


   
\subsection{Our class of Borcherds algebras}   
   We recall the following result of \cite{JurJPAA}. 
\begin{theorem}\label{JurThm}
 Let $A$ be a matrix satisfying conditions {\bf B1-B3}. Let $J$ and
$\mathfrak g_J$ be as above. Assume 
that if $i,j \in I \backslash J$ and $i \neq j$ then $a_{ij}<0$. 
Then $$\mathfrak g = \mathfrak u^+ \oplus (\mathfrak g_J + \mathfrak h) \oplus \mathfrak u^-,$$ 
where
$\mathfrak u^- = L(\coprod_{i \in I\backslash J}{ U}(\mathfrak n^-_J)\cdot f_i) $
and
$\mathfrak u^+= L(\coprod_{i \in I\backslash J}{ U}(\mathfrak n^+_J)\cdot e_i ) $.
For all $i \in I\backslash J$ the vector spaces 
$V^\prime_i={ U}(\mathfrak n^-_J)\cdot f_i $ is an integrable highest weight $\mathfrak g_J$-module
 and $V_i={U}(\mathfrak n^+_J)\cdot e_i $ 
is an integrable lowest weight $\mathfrak g_J$-module. 
\label{thm:free}
\end{theorem}

In particular, each $V^\prime_i$ is the unique (up to isomorphism) irreducible $\mathfrak{g}_J$-module of highest weight  $\alpha_i$. Let

\begin{equation}
V= \coprod_{i \in I\backslash J}{ U}(\mathfrak n^+_J)\cdot e_i \label{HW}
\end{equation}
\begin{equation} 
 V^\prime =  \coprod_{i \in I\backslash J}{ U}(\mathfrak n^-_J)\cdot f_i . \label{LW}
 \end{equation}

Let $S$ be a basis for integrable module $V$, and similarly let $S^\prime$ a basis for the integrable module $V^\prime$. By definition $L(V)=L(S)$ and $L(V^\prime) = L(S^\prime)$.  We use both notation, depending on whether we are focusing on the $\mathfrak g_J$-module structure or the generating sets.

Taking the universal enveloping algebra we obtain a decomposition as vector spaces and $\mathfrak h^e$-submodules:

 \begin{equation} U(\mathfrak g) = U(\mathfrak u^- \oplus (\mathfrak g_J + \mathfrak h) \oplus \mathfrak u^+) \cong U(\mathfrak u^-) \otimes U(\mathfrak g_J + \mathfrak h) \otimes U(\mathfrak u^+) \end{equation}

alternatively

 \begin{equation} U(\mathfrak g) \cong  T(V^\prime) \otimes U(\mathfrak g_J + \mathfrak h) \otimes T(V)  \label{decomp}
 \end{equation}


The following corollary follows immediately from Theorem \ref{thm:free}.

\begin{corollary}\label{integrable}
     The tensor algebra $T(V')$ over $\K$ is a $ Q$-graded integrable $\mathfrak{g}_J$-module with finite dimensional homogeneous subspaces. 
         
  \label{TVisintegrable}
\end{corollary}

Proof: Direct sums and tensor products of integrable modules are integrable. Hence the $\mathfrak{g}_J$-module $V^\prime$ (respectively $V$) is an integrable $\mathfrak{g}_J$-module, graded by $\Delta_{-}$ (respectively $\Delta_{+}$) with finite dimensional homogeneous subspaces, as is the algebra $T(V^\prime)$, when endowed with the usual action on tensor products, and extending linearly. 
 \qed

We denote by \( \theta \) the Chevalley involution of \( \mathfrak{g} \), defined by:
\[
\theta(e_\alpha) = -f_\alpha, \quad \theta(f_\alpha) = -e_\alpha, \quad \theta(h_\alpha) = -h_\alpha.
\]

This involution establishes an isomorphism as modules between $\mathfrak u^+$ and $\mathfrak u^-$.

\subsection{Magnus group for the free Lie algebra}\label{Magnus}

Let $X$ be a non-empty set. 
    Let $A(X)=A_{\K}(X)$ denote the free associative algebra over $\K$. Let $\widehat{A}({X})=\widehat{A}_\K({X})$ be the formal completion of ${A}({X})$. The completion $\widehat{A}(X)$ is called the {\it Magnus algebra} on $X$ and consists of formal power series in the (non-commuting) variables identified with the elements of ${X}$. 

 

Let $w$ denote a word in ${X}$ and let $c_w\in\K$. For $u=\sum_{w} c_w w\in \widehat{A}({X})$, set $$u_n=\sum_{|w|=n} c_w w.$$ Let $u=\sum_{n=0}^{\infty} u_n$. If $v_n=\sum_{|w|=n} c_w w\in \widehat{A}({X})$ then
$$(uv)_n=\sum_{m=0}^{\infty} u_mv_{n-m}\in \widehat{A}({X}).$$
This defines an $\N$-grading on $\widehat{A}(X)$.  We define the Lie subalgebra $\widehat{L}(X)=\widehat{L}_\K(X)$ in $\widehat{A}(X)$ which consists of the series $u= \sum_{n=1}^{\infty} u_n$ where $u_n \in L_n(X)$, $n = 1,2, \ldots$.

The elements $u \in \widehat{L}(X)$ are series without constant term, so 
$$ \exp (u)  ,\  \ln(1+u)$$
are well defined.

The group $M(X)=M_\K(X)$ is defined as the subgroup of invertible elements of $\widehat{A}({X})$ which are the formal power series with constant term 1. The group $M(X)$ is sometimes called the {\it Magnus group} \cite{Ba93}. In this paper, we follow the convention of other authors (for example \cite{MR1203518} and \cite{MR0422434}) and reserve the term `Magnus group' for the subgroup 
\begin{equation}
G(X) := \exp( \widehat{L}(X))\subset M(X).\label{magnus}\end{equation}   
The group $\exp( \widehat{L}(X))$ is
the subgroup of $M(X)$ generated by exponentials of  Lie words in the completion $\widehat{L}(X)$.


\section{Group construction}

The semi-direct product group $G = G(S') \rtimes G_J$ we construct in this section offers an alternative construction to the  Borcherds--Kac--Moody groups developed via Tits-style generators and relations in \cite{CJM}.
 Our  approach is uniquely suited to algebras admitting the decomposition $\mathfrak{g} = \mathfrak{u}^+ \oplus (\mathfrak{g}_J + \mathfrak{h}) \oplus \mathfrak{u}^-$ of \cite{JurJPAA}, where $\mathfrak{g}_J$ is a Kac-Moody (or semi-simple) subalgebra and $\mathfrak{u}^\pm$ are free Lie algebras generated by imaginary root vectors. This is precisely the class of Borcherds algebras that have no mutually orthogonal imaginary simple roots.

We assume that we have  a Lie algebra satisfying conditions of Theorem~\ref{JurThm}, that is, 
$A$ is a matrix satisfying conditions {\bf B1-B3}. Assume also that if $i,j \in I \backslash J$ and $i \neq j$ then $a_{ij}<0$. In particular, the $\alpha_{i}, \alpha_j $ are not orthogonal if $i\neq j$.

Making use of the decomposition $\mathfrak g = \mathfrak u^+ \oplus (\mathfrak g_J + \mathfrak h) \oplus \mathfrak u^-$,
we will  construct a group $G$ associated to the subalgebra $ \mathfrak g_J  \oplus \mathfrak u^-$, where $\mathfrak u^- = L(V^\prime)$. As vector spaces and $\mathfrak{h}$-modules we have

$$ U( (\mathfrak g_J + \mathfrak h) \oplus \mathfrak u^-) = U( \mathfrak g_J + \mathfrak h) \otimes U( \mathfrak u^-).$$
 We first construct a group $G_J$ associated  to the semi-simple or Kac--Moody algebra $\mathfrak g_J$ using the action on the integrable module $ U(\mathfrak u^-) = T(V^\prime)$. We then choose a basis $S^\prime$ of $V^\prime$ and form the Magnus group  $G(S^\prime)=\exp( \widehat{L}(S^\prime))$
where $\widehat{L}(S^\prime)$ is a completion of the free subalgebra $\mathfrak u^-$.
The  desired group $G$ is then naturally a semi-direct product of $G(S^\prime)$ and $G_J$, imposing relations induced from the action of $(\mathfrak g_J + \mathfrak h)$ on $T(V^\prime)$.


Fix an ordering of $S^\prime$ and fix a Poincar\'e--Birkhoff--Witt basis of $ U (\mathfrak {g}_J^-)$. Hence the basis for each of the irreducible highest weight ${\mathfrak g}_J$ modules is of the form $(\ad s_1\cdots   \ad s_n) (f_i)$ $i\in I\backslash J$. We use multi-indices to indicate the basis element of $U(\mathfrak{g}_J^-)$ along with  each of the highest weight vectors $f_i\in S^\prime$, $i \in I\backslash J$.

\subsection{The Kac--Moody group $G_J$ associated to ${\mathfrak g}_J$ and the module $W$} \label{KMgroup} 


We construct the Kac--Moody group $G_J$ associated with ${\mathfrak g}_J$ over $\mathbb K$, relative to the module $W= T(V^\prime)= U(\mathfrak {u}^-)$ (see \cite{CG} and Section 6.1 of \cite{MP95}). 

The tensor algebra ${\mathfrak g}_J$-module $W= T(V^\prime)$ is an integrable module, (see Lemma \eqref{TVisintegrable}).  In particular the elements of $({\mathfrak g}_J)_\alpha $ act locally nilpotently for all $\alpha \in \Delta_{re}$.

In the notation of Section 6.1 of \cite{MP95}, we set $$ \widehat{G}_J= *_{\alpha \in \Delta_{re}}{(\mathfrak g}_J)_\alpha .$$ 

Then  $\exp: {(\mathfrak g}_J)_\alpha \rightarrow \GL(W)$ is a group homomorphism. 
Let $$\pi: \mathfrak{g}_J \rightarrow \mathfrak{gl}(W )$$ be a representation of $\mathfrak{g}_J$ on $W$.
For all real $\alpha\in \Delta_J$ and all elements $x_\alpha \in  ({\mathfrak g}_J)_\alpha  $, 
$\exp (\pi( x_\alpha) ) $ is a well defined element of $\GL(W)$. 

There is a unique induced homomorphism 
$ \widehat \pi : \widehat G_J \rightarrow \GL(W)$ such that the following diagram commutes:

\begin{center}
\begin{tikzpicture}[every node/.style={midway}]
\matrix[column sep={6em,between origins},
        row sep={4em}] at (0,0)
{ \node(S)   {$\mathfrak g_\alpha$}  ; & \node(F) {$\GL(W)$}; \\
  \node(G) {$\widehat{G}_J$};                   \\};
\draw[->] (S) -- (G) node[anchor=east]  {$\exp_{\widehat{\alpha}}$};
\draw[->] (G) -- (F) node[anchor=north]  {$\widehat{\pi}$};
\draw[->] (S)   -- (F) node[anchor=south] {$\exp^{\pi}_\alpha$};
\end{tikzpicture}
\end{center}
Let $K= \ker \widehat \pi\leq \widehat G_J$, and let $$G_J = \widehat G_J \slash K.$$

 Let $i\in I\backslash J$, consider the module $V^\prime =  \coprod_{i \in I\backslash J}{ U}(\mathfrak n^-_J)\cdot f_i $, then $V'$ is integrable as is $W= T(V^\prime)$. Hence the  action of $\mathfrak g_J$ is locally nilpotent on $W$ for each $\alpha \in \Delta_J^{re}$. Thus we have relations
$$\pi(\mathfrak g_{\alpha})^n(v)=0$$ for all $\a\in \Delta_J^{\re}$ and for all $v\in W$, with sufficiently large $n$ depending on $v$.
In particular for all $x_\alpha\in \mathfrak g_{\alpha}$,
$$\pi(x_\alpha)^n(v)=0,
\text{ for some $n=n(v)$ and for all  $v\in W$}.$$

The group  $G_J$ is generated by the $\exp(\pi(x_\alpha))$ where  $x_\alpha$ is a root vector corresponding to the real root $\alpha \in \Delta_J^{\re}$. For all $x\in\mathfrak g_J$ and for all $g\in G_J$ we have
\begin{equation}\label{Adjoint}
\pi({\rm Ad}(g)(x))=\pi(g)\pi(x)\pi(g)^{-1}
\end{equation}
where ${\rm Ad}:G_J\to \Aut(\mathfrak g_J)$ (see \cite{MP95}, page 489). Equation (7) gives:
    \[ \pi(\rm{Ad}(g)(x)) = \pi(g)\pi(x)\pi(g)^{-1}. \]
Thus
 \begin{align*} \exp(\pi(\rm{Ad}(g)(x))) &= \exp(\pi(g)\pi(x)\pi(g)^{-1}) \\
&= \pi(g) \exp(\pi(x)) \pi(g)^{-1} \\
&=\rm{Ad}(\pi(x))(\exp{\pi(x)}).
\end{align*}


 \subsection{Defining an action of $\mathfrak{g}_J$ on $T(V^\prime)$ as formal polynomials. }\label{Polynomials}

In order to associate a group to $\mathfrak{g}(A)$ over $\K$, our first step is to transfer the action of $\mathfrak{g}_J$ to the setting where elements of $s\in S^\prime$ are treated as distinct non-commuting formal variables. We may then define an action of $G_J$ on $G(S^\prime)$.

First recall by Corollary \ref{TVisintegrable} that $T(V^\prime)$ is an integrable  $\mathfrak{g}_J$-module and denote the representation by 
$$\pi : \mathfrak{g}_J \mapsto \mathfrak{gl}(T(V^\prime)).$$

Fix a basis $S_i$ for each of the highest weight modules ${V_i}^\prime$ appearing in Theorem \ref{JurThm} and  let $S^\prime= \cup_i S_i$ denote the resulting basis of $V^\prime$. For example, we may choose basis elements corresponding to a Poincar\'e--Birkoff--Witt basis of $\mathfrak{n}^-_J$ as needed.

The tensor algebra $T(V')$ can  naturally be identified with  $T(S^\prime)$, viewed as an algebra of formal polynomials in non-commuting variables taken from the set $S^\prime$. Thus $T(S^\prime)$ is naturally isomorphic to the free associative algebra $A(S^\prime)$. We also have $\widehat{A}(S^\prime)\cong \widehat{T}(S^\prime)$.

\begin{proposition}\label{extend}
    There is  an action of $\mathfrak{g}_J$ as derivations on $T(S^\prime)$ when $S^\prime$ is considered as a set of non-commuting variables.
\end{proposition}  
\begin{proof}
Given $x\in \mathfrak{g}_J$ and any element $v\in V^\prime$ we can express $\pi(x)v= x \circ v\in V'$ in the basis $S^\prime$ with the same coefficients and call this element $x_v$. 

We define the action of $\mathfrak{g}_J$ on $T(S^\prime)$: We fix a vector space isomorphism between $V'$ and the span of the set $S'$. This induces a unique algebra isomorphism $\Psi: T(V') \to A(S')$. The derivation action of $\mathfrak{g}_J$ on $A(S')$ is defined by transferring the action via $\Psi$.


This element is now a polynomial in the formal variables from $S^\prime$.
\end{proof}
If we wish to be explicit, we can fix a Poincar\'e--Birkoff--Witt basis for $T(S^\prime)$, then 
the action of $x$ is defined on a monomial  $w = S^\prime_{i_1} S^\prime_{i_2} \cdots S^\prime_{i_\K}$, for $S^\prime_{i_j} \in S^\prime$ via the usual tensor product action and linearity:

\[
x \circ w := \sum_{j=1}^k S^\prime_{i_1} S^\prime_{i_2} \cdots (x \circ S^\prime_{i_j}) \cdots S^\prime_{i_\K}.
\]
In this sum, each term $(x \circ S^\prime_{i_j})$ is replaced by its polynomial expression $x_{S^\prime_{i_j}}$ in the basis $S^\prime$ considered as formal variables. The expression $x \circ w$ then becomes a polynomial in the formal variables from $S^\prime$.
 We then extend the action to all of  $T(S^\prime)$ linearly.


From Theorem \ref{JurThm}, these monomials  act on highest weight vectors \( f_i \in \mathfrak{g}_{-\alpha_i} \), \( i \in I \setminus J \) and determine bases for the $\mathfrak{g}_J $-modules \( V'_i \).

\subsection{Action of $G_J$ on $\widehat{T}(V')$}  Recall that we have a representation $\pi: \mathfrak{g}_J \to \mathfrak{gl}(T(V'))$. Let $S^\prime$ be a basis for $V'$, viewed as non-commuting formal variables.  In this subsection, we extend Proposition~\ref{extend}  to an action of exponentials of elements from $\pi(\mathfrak{g}_J)$ on the generating set of $G(S^\prime)$.
 

\begin{theorem}\label{GJaction}
$G(S^\prime)$ is a $G_J$-module.
\end{theorem}

\begin{proof}
 For $x \in \mathfrak{g}_J$, we extend the action of $\pi(x)$ on $V'$
to an action of $\exp(\pi(x))$ on $\widehat{A}(S^\prime)\cong \widehat{T}(S^\prime)$.

If $x = x_\alpha \in \mathfrak{g}_J^\alpha$ for a real root $\alpha \in \Delta_J^{\text{re}}$, the operator $\pi(x_\alpha)$ acts locally nilpotently on $T(V')$.  This  extends to a locally nilpotent action on each $Y \in \widehat{A}(S^\prime)$.

If $x = h \in \mathfrak{h}_J$,  $\pi(h)$ acts diagonally on weight vectors $w \in T(S^\prime)$  and preserves the grading in $T(S^\prime)$. In this case, $\exp(\pi(h)) \cdot w = e^{\lambda(h)}w$.  Thus  $\exp(\pi(x))$ is an algebra automorphism of $\widehat{A}(S^\prime)$. 
 This defines a homomorphism $\varphi: G_J \to \text{End}(T(S^\prime))$.


 Since $\mathfrak{g}_J$ acts linearly on $V'$, the action of $G_J$ preserves the grading of $T(V')$.

This action of $G_J$ on $T(V^\prime)$ can be extended to an action on the completion $\widehat{T}(V^\prime)$  since the automorphisms in $G_J$ are continuous with respect to the topology on $\widehat{T}(V^\prime)$.

We claim that this action restricts to an action of $G_J$ on $G(S^\prime)$. 

We show that this action preserves the subgroup $G(S^\prime) = \exp(\widehat{L}(S^\prime))$.
The action of $\pi(x)$ for any $x \in \mathfrak g_J$ on $\widehat{A}(S^\prime)$ preserves the Lie bracket structure and the $\N$-grading, so it preserves $\widehat{L}(S^\prime)\subset \widehat{A}(S^\prime)$. It follows that the action of any $g \in G_J$  also preserves the subspace $\widehat{L}(S^\prime)$.

 Let $g \in G_J$ and $\exp(L) \in G(S^\prime)$. Since the action of $g$ is by  algebra automorphisms, we have $g \cdot \exp(L) = \exp(g \cdot L)$. Since $g \cdot L \in \widehat{L}(S^\prime)$, $\exp(g \cdot L)\in G(S^\prime)$. The action $g \cdot \exp(L) = g \exp(L) g^{-1}$ defines a homomorphism $\varphi: G_J \to \text{Aut}(G(S^\prime))$.

\end{proof}

 We may write $g\exp(L) g^{-1} $ as ${\rm Ad}(g)(\exp (L))$. By \cite{MP95}, we have 
 $${\rm Ad}(g)(\exp (L))=\exp(\text{Ad}(g)(L)).$$

\subsection{The semi-direct product $G$}\label{SDproduct}

By Theorem~\ref{GJaction}, there is a homomorphism $\varphi: G_J\to \Aut(G(S^\prime))$. Thus we may construct the semi-direct product $$G=   G(S^\prime) \rtimes G_J $$
with group multiplication:
\[
(n_1, g_1)(n_2, g_2) = \left(n_1 \cdot \varphi(g_1)(n_2),\; g_1 g_2\right).
\]
for all  $n_i\in G(S^\prime)$ and $g_i\in G_J$.

The Kac--Moody group $G_J$, associated with the Lie algebra $\mathfrak{g}_J$, is generated by elements of the form $\exp(\pi(x_\alpha))$ where $\pi: \mathfrak{g}_J \to \mathfrak{gl}(T(V'))$ and  for a real root $\alpha \in \Delta_J^{\text{re}}$  we have $x_\alpha\in (\mathfrak{g}_J)_\alpha$.

For $g\in G_J$, $\varphi(g)\in\Aut(G(S^\prime))$.
By Theorem~\ref{GJaction}, the action of $g$ on an element $\exp(L) \in G(S^\prime)$,  for $L \in \widehat{L}(S^\prime)$, is given by $$\varphi(g)(\exp(L))=\exp(\pi(x)) \cdot \exp(L).$$  
Writing $g=\exp(\pi(x))\in G_J$, we have
$$\varphi(g)(\exp(L))=g \cdot \exp(L)=g\exp(L)g^{-1}.$$

The group $G(S^\prime)$ can be identified with the subgroup $\{(m, 1_{G_J}) \mid m \in G(S^\prime)\}$ in $G$. The group $G_J$ can be identified with the subgroup $\{(1_{G(S^\prime)}, g) \mid g \in G_J\}$ in $G$. With these identifications, $G(S^\prime)$ forms a normal subgroup of $G$.
The quotient group $G / G(S^\prime)$ is isomorphic to $G_J$. 

\begin{lemma}\label{mult} Multiplication in $G$ is given by
\[
(n_1, g_1) \cdot (n_2, g_2) = (n_1 \cdot \exp(\ad(g_1)(L)), g_1 g_2)
\]
where $L = \log(n_2) \in \widehat{L}(S^\prime)$.
\end{lemma}

\begin{proof} The group multiplication is  $$(n_1 \cdot \varphi(g_1)(n_2), g_1 g_2)$$
where the action is given by $\varphi(g_1)(n_2) = \Ad(g_1)(n_2) = g_1 n_2 g_1^{-1}$. Since $n_2 \in G(S^\prime) = \exp(\widehat{L}(S^\prime))$, there is a unique $L \in \widehat{L}(S^\prime)$, namely $L=\log(n_2)$, such that $n_2 = \exp(L)$. 
We have

    \[
    \varphi(g_1)(n_2) = \Ad(g_1)(n_2) = \Ad(g_1)(\exp(L)) = \exp(\Ad(g_1)(L)).
    \]

Thus
    \[
    (n_1, g_1) \cdot (n_2, g_2) = (n_1 \cdot \exp(\Ad(g_1)(L)), g_1 g_2)
    \]
    where $L = \log(n_2)$.

\end{proof}

Note that in Lemma~\ref{mult},  the exponential map is a bijection from $\widehat{L}(S^\prime)$ to $G(S^\prime) = \exp( \widehat{L}(S^\prime))$. This guarantees the existence and uniqueness of $L=\log(n_2)\in \widehat{L}(S^\prime)$.

 In summary we have the following. \begin{theorem}\label{GroupConstr}
 Let $A$ be a matrix satisfying conditions {\bf B1-B3}. Let $J$ and
$\mathfrak g_J$ be as above. Assume 
that if $i,j \in I \backslash J$ and $i \neq j$ then $a_{ij}<0$. 
Then associated to the decomposition $$\mathfrak g = \mathfrak u^+ \oplus (\mathfrak g_J + \mathfrak h) \oplus \mathfrak u^-,$$ 
there is a group $$G=   G(S^\prime) \rtimes G_J $$
where $S^\prime$ is a  basis for $V^\prime$,  
$ V'=\coprod_{i \in I\backslash J}{ U}(\mathfrak n^-_J)\cdot f_i$, $\mathfrak  u^-=L(V')$, 
  $G_J$ is a Kac--Moody group of $\mathfrak g_J$ relative to a representation of $\mathfrak g_J$ on the tensor algebra $W= T(V^\prime)$, 
 and $G(S^\prime)= \exp(\widehat{L}(S^\prime))$. 
If $g = \exp(\pi(x))$ for some $x\in \mathfrak g_J$ then the action of   $G_J$ on $G(S^\prime)$ is given by
$$g\cdot \exp(L)=g \exp(L) g^{-1}={\rm Ad}(g)(\exp (L))=\exp(\rm{Ad}(g)(L))$$
where $L\in \widehat{L}(S^\prime)$. Thus the action of $G_J$ on $G(S^\prime)$ is by Lie algebra automorphisms on $\widehat{L}(S^\prime)$, which then extends to group automorphisms on $G(S^\prime)$. Multiplication in $G$ is given by $$ (n_1, g_1) \cdot (n_2, g_2) = (n_1 \cdot \exp(\text{{\rm Ad}}(g_1)(L)), g_1 g_2) $$
for $L=\log(n_2)\in \widehat{L}(S^\prime)$.

\end{theorem}

The following theorem shows that the group $G$ is independent of the chosen basis $S^\prime$ for $V^\prime$, up to isomorphism.
\begin{theorem}\label{independent}
Let $S_1'$ and $S_2'$ be distinct bases for $V'$. Let $G_1 = G(S_1') \rtimes G_J$ and $G_2 = G(S_2') \rtimes G_J$. Then $G_1$ and $G_2$ are isomorphic.
\end{theorem}

\begin{proof}
We first construct an isomorphism $\Phi: G(S_1') \to G(S_2')$ induced by the change of basis from $S_1'$ to $S_2'$.
Let $A(S_1')$ and $A(S_2')$ be the free associative algebras on the sets of formal variables $S_1'$ and $S_2'$, respectively. The (invertible) change of basis map $\rho: S_1'\to S_2'$ extends uniquely to an algebra isomorphism:
\[
\widehat{\phi}: \widehat{A}(S_1') \to \widehat{A}(S_2')
\]
 defined on generators $s_i \in S_1'$, viewing the image of each $s_i$ as a formal polynomial. The isomorphism  $\widehat{\phi}$ induces an  isomorphism on the completions of free Lie algebras
\[
 \widehat{L}(S_1') \to \widehat{L}(S_2'),
\]
which in turn induces a group isomorphism between the corresponding Magnus groups:
\[
\Phi: G(S_1') \to G(S_2') \quad \text{defined by} \quad \Phi(\exp(L)) = \exp(\widehat{\phi}(L))
\]
for any $L \in \widehat{L}(S_1')$. 

We next show that $\Phi$ is compatible with the $G_J$-action. The actions of $G_J$ on  $G(S_1^\prime)$ and $G(S_2^\prime)$ are both induced from the representation $\pi: \mathfrak g_J \to \mathfrak{gl}(T(V^\prime))$. These actions are independent of the choice of basis for $V^\prime$.

The change of basis map $\rho$ commutes with the representation $\pi$: 
$$\rho(\pi(x)(v)) = \pi(x)(\rho(v)) \quad \text{for all } v \in V'.
$$
It follows that  the isomorphism $\widehat{\phi}$ commutes with the action of $\mathfrak g_J$:
\[
\widehat{\phi}(x \cdot L) = x \cdot (\widehat{\phi}(L)) \quad \text{for all } x \in \mathfrak g_J, L \in \widehat{L}(S_1').
\]
This extends to the action of $G_J$. For any $g \in G_J$ and $n \in G(S_1')$:
\[
\Phi(g \cdot n) = \Phi(\Ad(g)(n)) = \Ad(g)(\Phi(n)) = g \cdot (\Phi(n))
\]
Hence $\Phi$ is an isomorphism of $G_J$-modules. We define an isomorphism of semi-direct products $\Psi: G_1 \to G_2$ by:
\[
\Psi((n_1, g)) = (\Phi(n_1), g) \quad \text{for } n_1 \in G(S_1'), g \in G_J.
\]
Since $\Phi$ is an isomorphism, $\Psi$ is  a bijection. We show that $\Psi$ is a group homomorphism.

Let $(n_1, g_1), (n_2, g_2)\in G_1$. Then
\begin{align*}
\Psi((n_1, g_1) \cdot (n_2, g_2)) &= \Psi( (n_1 \cdot (g_1 n_2 g_1^{-1}), g_1 g_2) ) \\
&= (\Phi(n_1 \cdot (g_1 n_2 g_1^{-1})), g_1 g_2) \\
&= (\Phi(n_1) \cdot \Phi(g_1 n_2 g_1^{-1}), g_1 g_2), && \text{since $\Phi$ is a homomorphism} \\
&= (\Phi(n_1) \cdot (g_1 \Phi(n_2) g_1^{-1}), g_1 g_2) && 
\end{align*}
which is multiplication  in $G_2$. Thus $\Psi$ is a group homomorphism and hence an isomorphism. 
\end{proof}

\section{Examples}\label{examples}

 We recall that throughout this section, the
$f_i$ are the Chevalley generators for the imaginary simple roots and 
$U(\mathfrak{n}_J^-) \cdot f_i$ is  the cyclic module generated by $f_i$.

\subsection{A Magnus group for the Monster Lie algebra}\label{MonsterEx}

Let  $\mathfrak g$ be the Borcherds algebra associated to the Borcherds Cartan matrix $A$:

\begin{equation}\label{mdec}%
{A= \begin{blockarray}{cccccccccc}
 & & \xleftrightarrow{c(-1)}  &   \multicolumn{3}{c}{$\xleftrightarrow{\hspace*{0.7cm}c(1)\hspace*{0.7cm}}$}   &  \multicolumn{3}{c}{$\xleftrightarrow{\hspace*{0.7cm} { c(2)}\hspace*{0.7cm}}$}   & \\
\begin{block}{cc(c|ccc|ccc|c)}
  &\multirow{1}{*}{$c(-1)\updownarrow$} & 2 & 0 & \dots & 0 & -1 & \dots & -1 & \dots \\ \cline{3-10}
    &\multirow{3}{*}{ $ \,\,c(1)\,\,\left\updownarrow\vphantom{\displaystyle\sum_{\substack{i=1\\i=0}}}\right.$}& 0 & -2 & \dots & -2 & -3 & \dots & -3 &   \\
      & & \vdots & \vdots & \ddots & \vdots & \vdots & \ddots & \vdots & \dots  \\
        & & 0 & -2 & \dots & -2 & -3 & \dots & -3 &   \\ \cline{3-10}
         & \multirow{3}{*}{ $\,\,c(2)\,\, \left\updownarrow\vphantom{\displaystyle\sum_{\substack{i=1\\i=0}}}\right.$}  & -1 & -3 & \dots & -3 & -4 & \dots & -4 &   \\
    && \vdots & \vdots & \ddots & \vdots & \vdots & \ddots & \vdots & \dots  \\
        && -1 & -3 & \dots & -3 & -4 & \dots & -4 &   \\ \cline{3-10}
         && \vdots &  & \vdots &  &  & \vdots & \vdots &   \\
\end{block}
\end{blockarray}}\;\;,
\end{equation}

The numbers $c(j)$ are the Fourier coefficients of $q^j$  in the modular function $J(q)$:

\begin{equation}\label{Jfun}J(q ) = \sum_{n= -1}^\infty c(n)q^n=\dfrac{1}{q} + 196884q +21493760q^2+864299970q^3+\dots,
\end{equation}
where  $q=e^{2\pi{\bf{i}}\tau}$ for $\tau\in\C$ with $\text{Im}(\tau)>0.$

The following proposition  gives explicit generators and relations for $\mathfrak{m}=\mathfrak  g_J/\mathfrak  z$, where $\mathfrak  z$ is the center of $\mathfrak  g_J$
\begin{proposition}\label{P-Monster} \cite{JurJPAA}
The Serre--Chevalley generators of $\mathfrak{m}$ over $\K$ are
$e_{-1}$, $f_{-1}$, $h_{1}$, $h_2$, and $e_{jk}$, $f_{jk}$ for all
$(j,k)\in  I-\left\{(-1,1) \right\}$,
with defining relations:
\begin{align*}
\tag{M:1}\label{Mhh} \left[h_{1},h_2\right]&=0,\\
\tag{M:2a}\label{Mhe-} \left[h_1,e_{-1} \right] &= e_{-1},&              
              \left[h_2, e_{-1} \right] &= -e_{-1},\\
\tag{M:2b}\label{Mhe} \left[h_1,e_{jk} \right]&= e_{jk},&
              \left[h_2, e_{jk} \right] &= j e_{jk},\\
\tag{M:3a} \label{Mhf-} \left[h_1,f_{-1} \right] &= -f_{-1},&
             \left[h_2,f_{-1} \right] &= f_{-1},\\
\tag{M:3b}\label{Mhf} \left[h_1,f_{jk} \right]&= - f_{jk},&
              \left[h_2,f_{jk} \right] &= -j f_{jk},\\
\tag{M:4a}\label{Me-f-} \left[e_{-1},f_{-1} \right]&=h_1-h_2, \\
\tag{M:4b}\label{Mef-} \left[e_{-1},f_{jk} \right]&=0,  & \left[e_{jk},f_{-1} \right]&=0,\\
\tag{M:4c}\label{Mef} \left[e_{jk},f_{pq} \right]&=-\delta_{jp}\delta_{kq} \left(jh_1 + h_2\right),\\ 
\tag{M:5}\label{Mee}
  \left(\ad e_{-1} \right)^j e_{jk}&=0,&\qquad \left(\ad f_{-1} \right)^j f_{jk}&=0,
\end{align*}
for all $(j,k),\,(p,q) \in  I-\left\{(-1,1) \right\}$. Also ${\mathfrak  h} = \C h_1\oplus \C h_2 = \mathfrak{h}_A/\mathfrak{z}$, the Cartan subalgebra of~$\mathfrak m$.
\end{proposition}
From \cite{JurJPAA}, we have
$$\mathfrak g = \mathfrak u^+ \oplus (\mathfrak g_J + \mathfrak h) \oplus \mathfrak u^-$$
and 
 $$\mathfrak m \cong\mathfrak g / \mathfrak z$$
where  $\mathfrak{z}$ is the center of $\mathfrak{g}(A)$.  We have $\mathfrak{g}_J= \mathfrak{sl_2}$ with standard basis \( \{e_{-1}, f_{-1}, h_{-1}\} \) satisfying:
\[
[h_{-1}, e_{-1}] = 2e_{-1}, \quad [h_{-1}, f_{-1}] = -2f_{-1}, \quad [e_{-1}, f_{-1}] = h_{-1}.
\]
Then \( \mathfrak{n}_J^- = \langle f_{-1} \rangle \) and \( U(\mathfrak{n}_J^-) = \mathbb{C}[f_{-1}] \). Applying Theorem \ref{thm:free} and indexing to reflect the block form of \eqref{mdec}
\begin{align*}
V^\prime&= \coprod_{j\in\N} U(f_{-1})\cdot f_{j k}\\
V &= \coprod_{j\in\N}{ U}(e_{-1})\cdot e_{j k}.
\end{align*}

Let ${\mathfrak g_J} = \mathfrak {sl}_2$. Here $V^\prime$ is spanned by $S^\prime =\cup_{j \in \mathbb N}
\{ (f_{-1})^\ell \cdot f_{jk}\mid\ 0 \leq \ell<j, 1 \leq k \leq
c(j) \}$ a set of  $\mathfrak {sl}_2$-highest weight modules with  $\mathfrak{u}^- = L(V^\prime)$ and $\mathfrak{u}^+ = L(V)$.

In this case, we can be more explicit about the structure of the tensor algebra $T(V)$. 
The irreducible $\mathfrak  m$-module $T(V)$, constructed in \cite{jlw},  has the structure of a generalized Verma module and is induced from an irreducible module for
the parabolic subalgebra $\mathfrak u^+\oplus \mathfrak {gl}_2$ of $\mathfrak  m$.



Let $f_{jk}\in \mathfrak g_{-\alpha_{jk}}$ be the Chevalley generator corresponding to the imaginary simple root $\alpha_{jk}$. Then $f_{jk} \in V'_{jk}$ is a highest weight vector. 
Since $V'_{jk}$ is 3-dimensional, $f_{jk}$ satisfies:

 \begin{align*}
     e_{-1} \cdot f_{jk} &= 0\\
   h_{-1} \cdot f_{jk} &= \lambda_{jk} f_{jk}\text{ (for some weight $\lambda_{jk}$)}\\
     (f_{-1})^3 \cdot f_{jk} &= 0.
   \end{align*}

    Each $V'_{jk}$ is a 3-dimensional irreducible highest weight $\mathfrak{g}_J = \mathfrak{sl}_2$-module. The module $V'_{jk}$ is generated by the action of $U(\mathfrak{n}_J^-)$ on the highest weight vector $f_{jk}$:
    \[
    V'_{jk} = U(\mathfrak{n}_J^-) \cdot f_{jk} = {\rm Span}_{\mathbb{C}}\{ (f_{-1})^\ell \cdot f_{jk} \mid (\ell,j,k)\in E\}.
    \]

Thus, the Poincar\'e--Birkoff--Witt basis for \( U(\mathfrak{n}_J^-) \cdot f_i \) is:
\[
\{ f_{-1}^\ell \cdot f_{jk} \}
\]

and 
\[
S^\prime = \bigcup_{j \in \mathbb{N}} \left\{ f_{-1}^\ell \cdot f_{jk} \mid 0 \leq \ell < j, 1 \leq k \leq c(j) \right\}
\]
is a Poincar\'e--Birkoff--Witt-basis for \( V' = \bigoplus V'_{jk} \).    The action of $U(\mathfrak{n}_J^-) = \mathbb{C}[f_{-1}]$ on $f_{jk}$ gives the basis:
    \begin{align*}
        v_1 &= f_{jk} \\
        v_2 &= f_{-1} \cdot f_{jk} \\
        v_3 &= (f_{-1})^2 \cdot f_{jk}
    \end{align*}
where $f_{-1}^3\cdot f_{jk}=0$, $e_{-1}\cdot f_{jk}=0$, $h_{-1}\cdot f_{jk} = \lambda f_{jk}$ for some $\lambda\in\C$. The formal power series non commuting variables are 
 \begin{align*}b_1&= f_{jk},
 \\  b_2&= f_{-1}\cdot f_{jk},\\  
 b_3&= (f_{-1})^2\cdot f_{jk}.
\end{align*}

We recall also that $U(\mathfrak{sl}_2)$ has a Poincar\'e--Birkoff--Witt basis of the form $ f^ih^ke^j$.  
We define an $\mathfrak{sl}_2$-action on the formal variables $b_1, b_2, b_3$ by transferring the above action:  for any element $y \in U(\mathfrak{sl}_2)$, its action $y \circ b_j$ is found by computing $y \cdot v_j$, expressing the result in the $\{v_j\}$ basis, and then substituting $b_j$ for $v_j$.

We have the correspondence with the basis vectors of $V'_{jk}$:
    \begin{align*}
        b_1 &\leftrightarrow v_1 = f_{jk} \\
        b_2 &\leftrightarrow v_2 = f_{-1} \cdot f_{jk} \\
        b_3 &\leftrightarrow v_3 = (f_{-1})^2 \cdot f_{jk}
    \end{align*}

The mapping $\mu: \text{Span}_{\mathbb{C}}(S^\prime) \to V'$ defined by $\mu(b_s) = v_s$, and extended linearly, is an isomorphism by construction.
 Each formal variable $b_s \in S^\prime$ is in one-to-one correspondence with a specific $v_s\in V'$. This is  a faithful representation of the module $V'$.



As an example, let  $x = (f_{-1})^2 + (f_{-1}e_{-1}) - (f_{-1}h_{-1})$.
We want to define $x \circ b_1$.

 The formal variable is $b_1$, corresponding to $v_1 = f_{jk}$.
 We compute $x \cdot v_1 = ((f_{-1})^2 + (f_{-1}e_{-1}) - (f_{-1}h_{-1})) \cdot f_{jk}$:
    \begin{itemize}
        \item $(f_{-1})^2 \cdot f_{jk} = v_3$
        \item $(f_{-1}e_{-1}) \cdot f_{jk} = f_{-1}(e_{-1} \cdot f_{jk}) = f_{-1}(0) = 0$
        \item $(f_{-1}h_{-1}) \cdot f_{jk} = f_{-1}(h_{-1} \cdot f_{jk}) = f_{-1}(\lambda_{jk} f_{jk}) = \lambda_{jk} (f_{-1} f_{jk}) = \lambda_{jk} v_2$
    \end{itemize}
    So, $x \cdot v_1 = v_3 + 0 - \lambda_{jk}v_2$.
 Transferring to formal variables:
    \[ x \circ b_1 = b_3 - \lambda_{jk}b_2. \]

The group $G_J=\SL_2$ is given by generators $e_{-1}$, $f_{-1}$, $h_{-1}(=h_1-h_2)$. Let $\widetilde{w} = \exp(e_{-1})\exp(-f_{-1})\exp(e_{-1})$.  
Let $u, v\in\K$ and $s,t\in\K^\times$. Then $\SL_2(\K)$ has  relations:

        \begin{align*}
            \exp(u e_{-1})\exp(v e_{-1}) &= \exp((u+v)e_{-1}) \\
            \exp(u f_{-1})\exp(v f_{-1}) &= \exp((u+v)f_{-1})\\
       \exp((\log s)h_{-1})\exp((\log t)h_{-1}) &= \exp((\log (st))h_{-1}) \\
       \exp(-t f_{-1})\exp(s e_{-1})\exp(t f_{-1})&=\exp(-t^{-1} e_{-1})\exp(-t^2s f_{-1})\exp(-t e_{-1})\\
            \exp((\log s)h_{-1}) \exp(u e_{-1}) (\exp((\log s)h_{-1}))^{-1} &= \exp(s u e_{-1}) \\
            \exp((\log s)h_{-1}) \exp(u f_{-1}) (\exp((\log s)h_{-1}))^{-1} &= \exp(s^{-1} u f_{-1})\\
            \widetilde{w} \exp(u e_{-1}) \widetilde{w}^{-1} &= \exp(-u f_{-1}) \\
            \widetilde{w} \exp(u f_{-1}) \widetilde{w}^{-1} &= \exp(-u e_{-1}) \\
            \widetilde{w} \exp((\log s)h_{-1}) \widetilde{w}^{-1} &= \exp((\log (s^{-1}))h_{-1})
        \end{align*}

Since $G(S^\prime) = \exp( \widehat {L}(S^\prime))$, every element of $G(S^\prime)$ is of the form $\exp(L)$ for $L\in \widehat {L}(S^\prime)$.
As in Theorem~\ref{GroupConstr}, the action of $g \in \SL_2(\K)$ on $\exp(L) \in G(S^\prime)$ is $g\cdot \exp(L)=\exp(\text{Ad}(g)(L))$, 
where $g = \exp(\pi(x))$ for some $x\in\mathfrak{sl}_2(\K)$.

Let $f_{\ell,jk}$ be a basis element of $S^\prime$.  The action of $\SL_2(\K)$ on $\exp(v f_{\ell,jk})$, for $v \in\K$, which defines the semi-direct product $G=G(S^\prime)\rtimes   \SL_2(\K)$ is as follows:

 The operators $\text{ad}_{e_{-1}}$ and $\text{ad}_{f_{-1}}$ act locally nilpotently on $\widehat{L}_\K(S^\prime)$. For any $Y \in S^\prime$, $\exp(\text{ad}_{ue_{-1}})(Y)$ is a finite $\K$-linear sum of Lie words in $S^\prime$. Thus, $\exp(\text{ad}_{ue_{-1}})$ and $\exp(\text{ad}_{uf_{-1}})$ are automorphisms of $\widehat{L}_\K(S^\prime)$.

If $f_{\ell,jk}$ is an eigenvector for $\text{ad}_{h_{-1}}$ with eigenvalue $\lambda_{\ell,jk}$, then for $v \in\K, S^\prime \in\K^\times$:
    \[ \exp((\log S^\prime)h_{-1}) \exp(v f_{\ell,jk}) (\exp((\log S^\prime)h_{-1}))^{-1} = \exp( v (S^\prime)^{\lambda_{\ell,jk}} f_{\ell,jk} ). \]
 
The action of $\exp(ue_{-1})$,  $u \in\K$ is described as follows:
    \[ \exp(ue_{-1})\exp(v f_{\ell ,jk})(\exp(ue_{-1}))^{-1}= \left(\prod_{\alpha\in R'(\ell,j,k)} \exp(c_{\alpha}(u,v) f_{\alpha}) \right) \exp(v f_{\ell ,jk}) \]
    where $f_{\alpha}$ is a fixed  basis element for the root space $\mathfrak{g}_\alpha$,  $\a=(a-b\ell)\alpha_{-1}-b\alpha_{jk}$ and $c_{\alpha}(u,v)$ is a constant depending on $\a$, $u$ and $v$.

 If $[e_{-1}, f_{0,jk}] = 0$, then $c_{\alpha}(u,v)=0$ for $\alpha \in R'(\ell,j,k)$ and
    \[ \exp(ue_{-1})\exp(v f_{0,jk})(\exp(ue_{-1}))^{-1}=\exp(v f_{0,jk}) \]

The action of $\exp(uf_{-1})$, $u \in\K$ is given as follows:
    \[ \exp(uf_{-1})\exp(v f_{\ell ,jk})(\exp(uf_{-1}))^{-1}=\left( \prod_{\alpha\in S^\prime(\ell,j,k)} \exp(d_{\alpha}(u,v) f_{\alpha})\right) \exp(v f_{\ell ,jk}) \]
    where $f_{\alpha}$ is a fixed basis element for $\mathfrak{g}_\alpha$, $\a=-(a+b\ell)\alpha_{-1}-b\alpha_{jk}$ and $d_{\alpha}(u,v)$ is a constant depending on $\a$, $u$ and $v$.

 If $[f_{-1}, f_{j-1,jk}] = 0$, then $d_{\alpha}(u,v)=0$ for $\alpha \in S^\prime(\ell,j,k)$ and 
    \[ \exp(uf_{-1})\exp(v f_{j-1,jk})(\exp(uf_{-1}))^{-1}=\exp(v f_{j-1,jk}). \]

The root sets $R'(\ell,j,k)$ and $S^\prime(\ell,j,k)$ are:
\begin{align*}
R'(\ell,j,k)&=\left\{ (a-b\ell)\alpha_{-1}-b\alpha_{jk} \mid a,b\in\N, a-b\ell<j,\ 1 \leq k \leq c(j) \right\} \\
S^\prime(\ell,j,k)&=\left\{ -(a+b\ell)\alpha_{-1}-b\alpha_{jk} \mid a,b\in\N, a+b\ell<j,\ 1 \leq k \leq c(j) \right\}
\end{align*}
The notation $\alpha_{jk}=\alpha_{0,jk}$ refers to a simple imaginary root, and $\alpha_{-1}$ is the real simple root.

The $f_{\alpha}$ and $f_{\ell ,jk}$ are viewed as elements of  $\widehat {L}(S^\prime)$ corresponding to root vectors for $\alpha$ and $\alpha_{\ell,jk}$ respectively.  The factorizations given  are non-trivial consequences of the Baker--Campbell-Hausdorff formula (see \cite{CJM}). They appear in  \cite{CJM} in a slightly different form.

\subsection{Monstrous Lie algebras of  Fricke type}
An element $g\in\mathbb{M}$ is called {\it Fricke} if  the McKay--Thompson series $T_{g}(\tau)$ is invariant
under the level $N$ Fricke involution $\tau \mapsto -1/N\tau$ for some $N\geq 1$.

Otherwise $g$ is called non-Fricke. There are 141 Fricke classes, and 53 non-Fricke classes in $\M$.

When $g$ is Fricke, $\mathfrak m_g$ has similar structure to the
Monster Lie algebra $\mathfrak m$.

The Lie algebra $\mathfrak m_g$, for $g$ of Fricke type, has denominator formula \cite{CarDuke}
$$f(p) - f(q^{1/N}) = p^{-1} \prod_{m >0, \; n \in \frac{1}{N}\Z} (1-p^m q^n)^{c_g(m,n)}$$ where $$f(q) = q^{-1} + \sum_{n=1}^\infty c_g(1,\frac{n}{N}) q^n.$$
The exponents $c_g(m,\frac{n}{N})$ are the multiplicities of roots $(m,\frac{n}{N})\in\Z\oplus \frac{1}{N}\Z$. 
They are  the coefficients of a discrete Fourier transform of the generalized McKay--Thompson series relative to the cyclic group $\langle g\rangle$.

For $g$ of Fricke type, we have $\mathfrak m_g=\mathfrak{g}(A_g)/\mathfrak z$ where $\mathfrak z$ is the center of $ \mathfrak{g}(A_g)$ and $A_g$ is the Borcherds Cartan matrix \cite{CarFricke}:

\begin{equation*}
{A_g=\smaller \begin{blockarray}{cccccccccc}
 & & \xleftrightarrow{c_g(1,\frac{-1}{N})}  &   \multicolumn{3}{c}{$\xleftrightarrow{\hspace*{0.7cm}c_g(1,\frac{1}{N})\hspace*{0.7cm}}$}   &  \multicolumn{3}{c}{$\xleftrightarrow{\hspace*{0.7cm} { c_g(1,\frac{2}{N})}\hspace*{0.7cm}}$}   & \\
\begin{block}{cc(c|ccc|ccc|c)}
  &\multirow{1}{*}{$c_g(1,\frac{-1}{N})\updownarrow$} & 2 & 0 & \dots & 0 & -1 & \dots & -1 & \dots \\ \cline{3-10}
    &\multirow{3}{*}{ $ \,\,c_g(1,\frac{1}{N})\,\,\left\updownarrow\vphantom{\displaystyle\sum_{\substack{i=1\\i=0}}}\right.$}& 0 & -2 & \dots & -2 & -3 & \dots & -3 &   \\
      & & \vdots & \vdots & \ddots & \vdots & \vdots & \ddots & \vdots & \dots  \\
        & & 0 & -2 & \dots & -2 & -3 & \dots & -3 &   \\ \cline{3-10}
         & \multirow{3}{*}{ $\,\,c_g(1,\frac{2}{N})\,\, \left\updownarrow\vphantom{\displaystyle\sum_{\substack{i=1\\i=0}}}\right.$}  & -1 & -3 & \dots & -3 & -4 & \dots & -4 &   \\
    && \vdots & \vdots & \ddots & \vdots & \vdots & \ddots & \vdots & \dots  \\
        && -1 & -3 & \dots & -3 & -4 & \dots & -4 &   \\ \cline{3-10}
         && \vdots &  & \vdots &  &  & \vdots & \vdots &   \\
\end{block}
\end{blockarray}}\;\;,
\end{equation*}

Let $$I = \left\{(j,k)\mid j\in\{-1,1,2,3,\dots\},\; 1 \leq k\leq c_g(1,j/N) \right\}. $$

For $(j,k)\in {I}$, the entries satisfy
$$ (\a_{jk},\a_{pq})=a_{jk,pq} = -(j+p). $$

In the case where $g$ is Fricke, Carnahan \cite{CarFricke} proved an analog for $\mathfrak{m}_g$ of Jurisich's decomposition $\mathfrak m\cong\mathfrak u^-\oplus\mathfrak{gl}_2\oplus\mathfrak u^+$ where $\mathfrak u^\pm$ are free Lie algebras. 
 Carnahan also showed that the Borcherds Cartan matrix for $\mathfrak{m}_g$ has the same form as that of $\mathfrak{m}$, with different block sizes as given by the dimensions of the root spaces for imaginary simple roots of $\mathfrak{m}_g$.

The $\mathfrak{gl}_2$ summand contains the subalgebra $\mathfrak{sl}_2$, which is our $\mathfrak{g}_J$. This $\mathfrak{sl}_2$ corresponds to the unique real simple root of $\mathfrak{m}_g$, which we label $\alpha_{-1}$. The associated group is $G_J = \mathrm{SL}_2(\mathbb{K})$.

The subalgebra $\mathfrak u^-$ is a free Lie algebra $L(V')$. The  space $V'$ is a direct sum of irreducible highest weight $\mathfrak{g}_J$-modules 
\[
V' = \bigoplus_{j \ge 1, \, 1 \le k \le c_g(1, \frac{j}{N})} V_{jk}.
\]

Each module $V_{jk}$ is generated by a highest weight vector $f_{jk}$ corresponding to an imaginary simple root $\alpha_{jk}$ of $\mathfrak{m}_g$. The  inner product is $(\alpha_{-1}, \alpha_{jk}) = -j$. The highest weight of the $\mathfrak{sl}_2$-module $V_{jk}$ with respect to $h_{-1}$ is therefore $\lambda = -(\alpha_{-1}, \alpha_{jk}) = j$. For $j \in \mathbb{N}$, this corresponds to an irreducible $\mathfrak{sl}_2$-module of dimension $j+1$.

 Let $S^\prime$ be the following basis of $V^\prime$: we form  the union of the Poincar\'e-Birkhoff-Witt basis elements for each of the $(j+1)$-dimensional modules $V_{jk}$:
\[
S^\prime = \bigcup \left\{ (f_{-1})^\ell \cdot f_{ik} \, \middle| \, j \ge 1, 1 \le k \le c_g(1, \frac{j}{N}), 0 \le \ell \le j \right\}
\]
 As in Subsection~\ref{MonsterEx}, we   construct $G = G(S^\prime) \rtimes G_J$
where $G(S^\prime) = \exp(\widehat{L}(S^\prime))$ and $\widehat{L}(S^\prime)$ is the completion of the free Lie algebra $\mathfrak u^- = L(S^\prime)$.

The $\mathrm{SL}_2(\mathbb{K})$ Group $G_J$
is generated by elements $\exp(ue_{-1})$, $\exp(uf_{-1})$, and $\exp((\log s)h_{-1})$, with the standard $\mathrm{SL}_2(\mathbb{K})$ relations as given in Subsection~\ref{MonsterEx}.

The group $G$ is defined as the semi-direct product $G = G(S^\prime) \rtimes \mathrm{SL}_2(\mathbb{K})$. The action of an element $h \in \mathrm{SL}_2(\mathbb{K})$ on an element $\exp(L) \in G(S^\prime)$ is given by
$h \cdot \exp(L)  = \exp(\mathrm{Ad}(h)(L))$.

The element $h_{-1} \in \mathfrak{sl}_2$ acts diagonally. The weight of $f_{\ell,jk}$ is $\lambda_{\ell,jk} = j - 2\ell$. The action for $s \in \mathbb{K}^\times$ and $v \in \mathbb{K}$ is:
    \[
    \exp((\log s)h_{-1}) \cdot \exp(v f_{\ell,jk}) = \exp(v(s^{j-2\ell}) f_{\ell,jk}).
    \]

    The action of $\exp(ue_{-1})$ on $f_{\ell,jk}$ is locally nilpotent and is given by the   Baker--Campbell--Hausdorff formula:
    \[
    \exp(ue_{-1}) \cdot \exp(v f_{\ell,jk}) = \left( \prod_{\alpha \in \R'(\ell,j,k)} \exp(c_\alpha(u,v) f_\alpha) \right) \cdot \exp(v f_{\ell,jk})
    \]
where  $c_{\alpha}(u,v)$ is a constant depending on $\a$, $u$ and $v$.    The root set $R'(\ell,j,k)$ consists of roots of the form $(a - b\ell)\alpha_{-1} - b\alpha_{jk}$ for $j \ge 1$, $1 \le k \le c_g(1, \frac{j}{N})$ and $0 \le \ell \le j$.

The action of $\mathrm{ad}\, f_{-1}$ on on $f_{\ell,jk}$ is also locally nilpotent and we have:
    \[
    \exp(uf_{-1}) \cdot \exp(v f_{\ell,jk}) = \left( \prod_{\alpha \in S^\prime(\ell,j,k)} \exp(d_\alpha(u,v) f_\alpha) \right) \cdot \exp(v f_{\ell,jk})
    \]
   where $d_{\alpha}(u,v)$ is a constant depending on $\a$, $u$ and $v$. The root set $S^\prime(\ell,j,k)$ consists of roots of the form $-(a + b\ell)\alpha_{-1} - b\alpha_{jk}$ for $j \ge 1$,  $1 \le k \le c_g(1, \frac{j}{N})$ and $0 \le \ell \le j$.

\subsection{Example: A Borcherds algebra from a hyperbolic Kac--Moody root lattice}

    One class of examples can be created by adjoining one or more simple imaginary roots to a semi-simple or Kac--Moody algebra. We consider the following  example, which is obtained from adjoining one simple imaginary root to the hyperbolic Kac--Moody algebra $H(3)$ which has  generalized Cartan matrix
    $$A = \begin{pmatrix}   2 & -3 \\ -3 &  ~2   \end{pmatrix}$$  
to obtain the Borcherds algebra $ \mathfrak g(B)$ corresponding to the rank $2$ matrix $$B = \begin{pmatrix}   2 & -3 &   -1 \\ -3 &  ~2 & -1  \\    -1 & -1& -2  \  \end{pmatrix}.$$  

 We denote the simple roots of the resulting Borcherds algebra as $\alpha_1, \alpha_2, \alpha_3$. Note that the matrix has rank 2, so we extend the Cartan subalgebra $\mathfrak{h}$ of $\mathfrak g(B)$ to include a degree derivation $D_3$ satisfying $D_3(\alpha_i) = \delta_{(3,i)}$. Let $J= \{ 1, 2\}$, so that $\mathfrak{g}_J= H(3)$. We write
$\mathfrak{g}_J=\mathfrak{n}_J^-\oplus \mathfrak{h}_J\oplus \mathfrak{n}_J^+$ for the triangular decomposition of $\mathfrak{g}_J$.

We note the following: when restricted to the Cartan subalgebra $\mathfrak{h}_J$, $\alpha_1$ and $\alpha_2$ are the simple roots of the hyperbolic algebra $H(3)$.
When restricted to the root lattice of $\mathfrak g_J=H(3)$, the root  $\alpha_{3} =   \alpha_1 +  \alpha_2$ is  an imaginary root in the root lattice of $H(3)$, but is a simple imaginary root of the Borcherds algebra $\mathfrak g(B)$.
As elements of $[(\mathfrak h)^e]^* = (\delta_{(3,i)} \ltimes \mathfrak h)^*$ the roots $\alpha_1, \alpha_2, \alpha_{3}$ are linearly independent.

We define $\mathfrak u^-$ to be the free Lie algebra $L(V')$ where $V'=U(\mathfrak n_J^-)\cdot f_{\a_3}$, where $f_{\a_3} $ is a root vector for $-\a_3$ in $\mathfrak g(B)$ and $\mathfrak u^+$ to be the free Lie algebra $L(V)$ where $V=U(\mathfrak n_J^+)\cdot e_{\a_3}$  and $e_{\a_3} $ is a root vector for $\a_3$ in $\mathfrak g(B)$.

The module $V'$ is an integrable highest weight module for  $\mathfrak g_J=H(3)$ with highest weight vector $f_{\alpha_3}$. 
Poincar\'e--Birkoff-Witt basis elements of  $V'$ are of the form $v = f_{i_1}^{m_1} f_{i_2}^{m_2}\dots f_{i_n}^{m_n} \cdot f_{\alpha_3}$ where $f_{i_j}^{m_j}\in \mathfrak n_J^-$.

Let $S^\prime$ a basis of $V^\prime$. We define $G(S^\prime)=\exp(\widehat{L}(S^\prime))$. As in \cite{Ti87},  the Kac--Moody group $G_J$ of  $H(3)$ is generated by the set 
$\{\chi_{\alpha}(u)\mid \alpha\in \Delta_{\re}, u\in\K\}$ satisfying Tits'
relations (R1)-(R7) below, where we identify $\chi_{\alpha}(u)$ with $\exp\ad (ue_\alpha)$ and $e_\alpha\in\mathfrak g_\a$ is a root vector corresponding to $\a\in \Delta_{\re}$. In (R1)-(R7) below, 
$i,j$ are elements of $I=\{1,2\}$,
$u,v\in\K$, $s,t\in\K^\times$
and $\alpha$ and $\beta$ are real roots.

(R1) $\chi_{\alpha}(u+v)=\chi_{\alpha}(u)\chi_{\alpha}(v)$;

(R2) Let $(\alpha,\beta)$ be a \it{prenilpotent pair, }\rm
that is, there exist $w,\ w'\in W$ such that
$$w\alpha, \ w\beta\in\Delta_{\re}^+{\text{ and }}w'\alpha, \ w'\beta\in\Delta_{\re}^-.$$
Then
\[
	[\chi_{\alpha}(u),\chi_{\beta}(v)]
	=
	\prod_{m,n \geq 1}
	\chi_{m\alpha+n\beta}(C_{mn\alpha\beta}u^m v^n)
\]
where the product on the right-hand side is taken over all real roots
of the form $m\alpha+n\beta$, $m,n\geq 1$, in some fixed order, and
$C_{mn\alpha\beta}$ are integers independent of $\K$.

Set

$\chi_{\pm i}(u)=\chi_{\pm \alpha_i}(u)$, where $\chi_{ i}(u)=\exp\ad(ue_i)$ and $\chi_{ -i}(u)=\exp\ad(uf_i)$

$\wgal_{i}(s)=\chi_{i}(s)\chi_{-i}(-s^{-1})\chi_{i}(s)$,

$\wgal_{i}=\wgal_{i}(1)$ and $h_{i}(s)=\wgal_{i}(s)\wgal_{i}^{-1}$.

\noindent The remaining relations are

(R3) $\wgal_{i}\chi_{\alpha}(u)\wgal_{i}^{-1}=
\chi_{w_{i}\alpha}(\eta_{\alpha,i}u)$,

(R4)
$h_{i}(s)\chi_{\alpha}(v)h_{i}(s)^{-1}=\chi_{\alpha}(vs^{\langle\alpha,\alpha_i^{
\vee}\rangle})$, 

(R5) $\wgal_{i}h_j(s)\wgal_i^{-1}=h_j(s)h_i(s^{-a_{ji}})$,

(R6) $h_{i}(st)=h_{i}(s)h_{i}(t)$, and

(R7) $[h_{i}(s),h_{j}(t)]=1$.

In \cite[Proposition 2.3]{KP}, it is shown that a pair
$\{\alpha,\beta\}$ is prenilpotent
if and only if $\alpha\neq -\beta$ and $|(\Z_{>0}\alpha +\mathbb
Z_{>0}\beta)\cap \Delta_{\re}|<\infty$.
Thus the product on the right-hand side of (R2) is  finite.

In Figure~\ref{hyp}, the hyperbola supporting the real roots of $H(3)$ in $\mathfrak h^*\cong \R^{1,1}$ is indicated in blue. As shown in \cite{CKMS}, a pair of real roots $\a,\b$ is prenilpotent if and only if $\a,\b$ both lie on the same branch of the blue hyperbola.

 {Recall that since the module $V'$ is integrable, the action of $e_\alpha$ and $f_\alpha$ are locally nilpotent on $V'$ for real roots $\alpha$ in $\mathfrak g_J$. Thus for $L \in \widehat{L}(S^\prime)$, the sums $\exp\ad(ue_\alpha)(L)$ and $\exp\ad(uf_\alpha)(L)$ terminate and are well-defined elements of $\widehat{L}(S^\prime)$.

}
 
  For $ h_i(t)\in G_J$, the  element  $h_i\in\mathfrak g$ acts diagonally on the basis vectors of $V'$, which are weight vectors. If $v \in V'$ is a basis vector of weight $\lambda$, then $h_i \cdot v = \langle\lambda, \alpha_i\rangle v$. This translates to the action on the corresponding formal variable $s_v$ as $h_i \circ s_v = \langle\lambda, \alpha_i\rangle s_v$. The group action is then:
    \[
    \Ad(h_i(t))(s_v) = t^{\langle\lambda, \alpha_i^\vee \rangle} s_v
    \]
    where $\alpha_i^\vee$ is the coroot corresponding to $\a_i$.


We may now construct the semi-direct product $G=   G(S^\prime) \rtimes G_J $ where 
$ V'={ U}(\mathfrak n^-_J)\cdot f_{\a_3}$. The group
  $G_J$ is the Kac--Moody group of $\mathfrak g_J=H(3)$ given by the presentation above 
 and $G(S^\prime)= \exp(\widehat{L}(S^\prime))$. 
 Multiplication in $G$ is given by Lemma~\ref{mult}.

\begin{figure}[h!]\begin{center}
\includegraphics[scale=0.57]{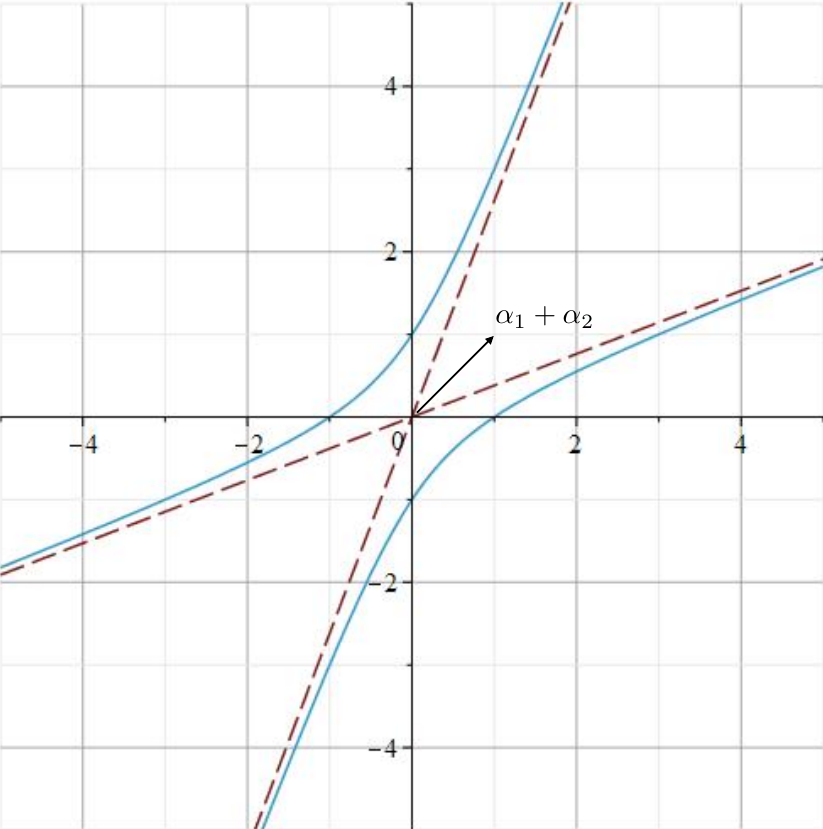}
\caption{The hyperbola supporting the real roots of $H(3)$ in $\mathfrak h^*\cong \R^{1,1}$ is indicated by the bold line curves. The asymptotes are indicated by the dashed lines. The imaginary root $\alpha_1+\alpha_2$ lies inside the imaginary cone.}\label{hyp}
\end{center}\end{figure}
In $H(3)$, the root $\alpha_1+\alpha_2$ is imaginary (see Figure~\ref{hyp}). Let 
Let $\frak{h}(H(3))$ be the Cartan subalgebra of $H(3)$. Recall that $\alpha_1$, $\alpha_2$ and $\alpha_3$ are the simple roots of  $\mathfrak g(A)$ and that $\alpha_3$ is imaginary.  We observe that when restricted to $\frak{h}(H(3))$, we have  $\alpha_{3} = \alpha_1 + \alpha_2$.









    Similar examples can be constructed by replacing the $-1$ entries in the generalized Cartan matrix $B$ with entries $-k$, $k>1$.

\subsection{A Borcherds algebra formed from $E_{10}$ and `missing modules'}


We next consider the 
Borcherds  algebra $\mathfrak{g}_{\rm{II}_{9,1}}$ of  \cite{GN95} and \cite{BGN98}, whose
maximal Kac–Moody subalgebra is the hyperbolic algebra $E_{10}$. 

The root lattice of $\mathfrak{g}_{\rm{II}_{9,1}}$ is the 10-dimensional even unimodular Lorentzian lattice $\rm{II}_{9,1}$. The root lattice of $E_{10}$ coincides with ${\rm II}_{9,1}$, however   $E_{10}$ is a certain subalgebra of $\mathfrak{g}_{\rm{II}_{9,1}}$ defined by a set of real simple roots, whereas the Borcherds algebra $\mathfrak{g}_{\rm{II}_{9,1}}$ has additional imaginary simple roots.




\begin{figure}
\begin{center}
\setlength{\unitlength}{1mm}
\begin{picture}(90,10)
   \put(0,0){\circle*{1.7}}
   \put(10,0){\circle*{1.7}}
   \put(20,0){\circle*{1.7}}
   \put(30,0){\circle*{1.7}}
   \put(40,0){\circle*{1.7}}
   \put(50,0){\circle*{1.7}}
   \put(60,0){\circle*{1.7}}
   \put(70,0){\circle*{1.7}}
   \put(80,0){\circle*{1.7}}
   \put(20,10){\circle*{1.7}}

   \put(0,0){\line(1,0){80}}
   \put(20,0){\line(0,1){10}}

   \put(-1.5,-4){$\alpha_{7}$}
   \put(8.5,-4){$\alpha_6$}
   \put(22,10){$\alpha_8$}
   \put(18.5,-4){$\alpha_5$}
   \put(28.5,-4){$\alpha_4$}
   \put(38.5,-4){$\alpha_3$}
   \put(48.5,-4){$\alpha_2$}
   \put(58.5,-4){$\alpha_1$}
   \put(68.5,-4){$\alpha_0$}
   \put(77.5,-4){$\alpha_{-1}$}
 
\end{picture}

\end{center}
\bigskip
  \caption{The Dynkin diagram of $E_{10}$}
  \label{fig:DynkinE10}
\end{figure}
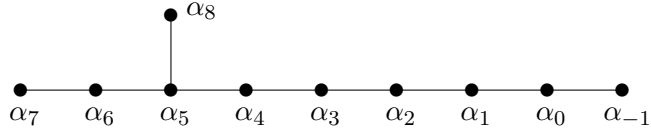

We first review the  details of the  $E_{10}$ root lattice. 
 Recall the $E_{10}$ Coxeter-Dynkin diagram shown in Figure \ref{fig:DynkinE10}. Throughout this section we will use the notation $E_{k}$ to refer to the $E_{k}$  lattice unless otherwise specified. The $E_{10}$ root lattice is $E_{10} = \oplus_{i=-1}^8 {\mathbb Z } \alpha_i$ where for all $-1 \leq i,j \leq 8$:

\begin{align*}
    (\alpha_i, \alpha_i) & = 2 \\
    (\alpha_i, \alpha_j) & = -1 \text{ if } i\neq j \text{ and the i-th and j-th node are joined by an edge} \\
     (\alpha_i, \alpha_i) & = 0 \text{ otherwise }.\\
\end{align*}

The vector 
 \[
\delta = \alpha_0 + 2\alpha_1 + 3\alpha_2 + 4 \alpha_3+ 5\alpha_4 + 6\alpha_5 + 4\alpha_6 + 2\alpha_7 + 3 \alpha_8.
  \]

is the null (also called isotropic) root of $E_9 \subset E_{10}$. For $i= 0, \ldots ,8 $ 
\[
(\alpha_i, \delta) = 0
\]
and 
\[
(\alpha_{-1}, \delta) = -1 .
\]

For a root $\lambda\in E_{10}$, we define its {\it level} $\ell$ to be $\ell= -\lambda \cdot \delta. 
$ The generalized Cartan matrix determined by the Dynkin diagram in Figure \ref{fig:DynkinE10} defines the hyperbolic Kac--Moody algebra $E_{10}$.

  The space of `missing modules' of  \cite{BGN98} coincides with the  subspace $M$ of $\mathfrak{g}_{\rm{II}_{9,1}}$. The subspace $M$ is the orthogonal complement of the extended Kac--Moody algebra $E_{10}$ with respect to the invariant bilinear form of $\mathfrak{g}_{\rm{II}_{9,1}}$, where $E_{10}$ is extended to incorporate the full Cartan subalgebra of $\mathfrak{g}_{\rm{II}_{9,1}}$. 
  
  The subspace $M$ is a module for $E_{10}$ under the adjoint action and $M$ can be uniquely decomposed into a direct sum of irreducible highest weight modules and/or lowest weight modules of $E_{10}$. Let  $\lambda$ be  a dominant integral weight (for highest weight modules) or a negative dominant integral weight (for lowest weight modules) of $\mathfrak{g}_{\rm{II}_{9,1}}$.

   The lowest weight vectors (or highest weight vectors ) that generate these irreducible $E_{10}$-modules in $M$ correspond to  specific imaginary simple roots of $\mathfrak{g}_{\rm{II}_{9,1}}$. These are the imaginary simple roots of $\mathfrak{g}_{\rm{II}_{9,1}}$ that are not simple root of the $E_{10}$ subalgebra.
   

We take the set of imaginary simple roots of $\mathfrak{g}_{\rm{II}_{9,1}}$ to be the set $\Lambda$ below. This is the set of level one imaginary simple roots in the fundamental Weyl chamber  as in  (\cite{BGN98} Proposition 1):

\[
\Lambda= \{\lambda_k = \alpha_{-1} +  (2+k) \delta \mid k \in \mathbb N\}.
\]


This gives rise to a   Borcherds Generalized Cartan matrix $A$ as follows:

\[ A=
\begin{bmatrix}
A_{E_{10}} & B \\
B^T & C
\end{bmatrix}
\]

where 

\[
A_{E_{10}} = \begin{bmatrix}
\begin{array}{ccccccccc|c}
2 & -1 & 0 & 0 & 0 & 0 & 0 & 0 & 0 & 0 \\
-1 & 2 & -1 & 0 & 0 & 0 & 0 & 0 & 0 & 0 \\
0 & -1 & 2 & -1 & 0 & 0 & 0 & 0 & 0 & 0 \\
0 & 0 & -1 & 2 & -1 & 0 & 0 & 0 & 0 & 0 \\
0 & 0 & 0 & -1 & 2 & -1 & 0 & 0 & 0 & 0 \\
0 & 0 & 0 & 0 & -1 & 2 & -1 & 0 & 0 & -1 \\
0 & 0 & 0 & 0 & 0 & -1 & 2 & -1 & 0 & 0 \\
0 & 0 & 0 & 0 & 0 & 0 & -1 & 2 & -1 & 0 \\
0 & 0 & 0 & 0 & 0 & 0 & 0 & -1 & 2 & 0 \\
\hline
0 & 0 & 0 & 0 & 0 & -1 & 0 & 0 & 0 & 2
\end{array}
\end{bmatrix}
\]

is the generalized Cartan matrix for the Kac--Moody algebra $E_{10}$,

\[ B= 
\begin{bmatrix}
0 & 1 & 2 & 3 & \ldots \\
-1 & -1 & -1 & -1 & \ldots\\
0 & 0 & 0 & 0 & \dots \\
\vdots & 
\end{bmatrix}
\]

and 
\begin{align*} C&= 
\begin{bmatrix}
-2 & -3  & -4 & -5 & \ldots \\
-3 & -4 &  -5 & -6  & \ldots\\
 -4 & -5 & -6 & -7 & \dots \\
\vdots & 
\end{bmatrix}. 
\end{align*}

The submatrices $B$ and $C$ are derived from 
lowest weight vectors (or highest weight vectors)    that generate $M$ and correspond to  specific imaginary simple roots of $\mathfrak{g}_{\rm{II}_{9,1}}$ that are not simple roots of $E_{10}$.

 As in the case of the Monster Lie algebra, the entries are in block form to reflect the dimensions of the simple imaginary root spaces. These are determined by the graded dimensions of the highest weight modules. A method for computing these multiplicities is derived from the fomula for the free Lie algebra and a method is given in \cite{GN95} where the multiplicities are denoted $ \mu(\Lambda)$. Note that here we are adjoining only the level one simple imaginary roots as their dimensions are determined. 
 

The Borcherds Cartan matrix $A$ satisfies the conditions of Theorem \ref{thm:free}. Let $\mathfrak{g}$ be the Borcherds algebra with the generalized Cartan matrix $A$. We set $\mathfrak{g}_J$ to be the Kac--Moody algebra $E_{10}$ generated by Chevalley generators $\{e_i, f_i, h_i\}_{i \in J}$, where $J = \{-1, 0, \ldots, 8\}$ corresponds to the real simple roots $\alpha_i$ for which $a_{ii}=2$.
       The subalgebra $\mathfrak{h}$ is the Cartan subalgebra of the full Borcherds algebra $\mathfrak{g}_{\rm{II}_{9,1}}$. It contains the Cartan subalgebra $\mathfrak{h}_J$ of $\mathfrak{g}_J$ and additional degree derivations $D$ such that $\mathfrak{h} = D \oplus \mathfrak{h}_J$. 
   
   We set $W = T(V')$
        where $V'$ is the direct sum of integrable highest weight $\mathfrak{g}_J$-modules: $V' = \bigoplus_{i \in I \setminus J} U(\mathfrak{n}_J^-) \cdot f_i$.
      
     The simple roots indexed by $I \setminus J $
     is the set $\{\lambda_k = \alpha_{-1} + (2+k)\delta \mid k \in \mathbb{N}\}$, which are the imaginary simple roots. Thus, $V' = \bigoplus_{k \in \mathbb{N}} V_{\lambda_k}$, where $V_{\lambda_k} = U(\mathfrak{n}_{E_{10}}^-) \cdot f_{\lambda_k}$ are irreducible highest weight $E_{10}$-modules with highest weight vector $f_{\lambda_k}$.

    Each $V_{\lambda_k}$, and thus $V'$, is  a $\mathfrak{g}_J$-module. The representation $\pi: \mathfrak{g}_J \to \mathfrak{gl}(T(V'))$ is the $\mathfrak{g}_J$-module action on $V'$ and respects the natural grading on $W$.


 The matrix $A$ satisfies conditions B1-B3. For distinct imaginary simple roots $\lambda_i, \lambda_j$ (corresponding to the $C$ block), $C_{ij} = i+j \neq 0$ for $i,j>0$. Hence no distinct imaginary simple roots are pairwise orthogonal, So Theorem 2.1 is applicable.

The group $G_J$ is a Kac--Moody group associated with $\mathfrak{g}_J = E_{10}$. This group is generated by elements of the form $\exp(\pi(x_\alpha))$ where $x_\alpha \in (\mathfrak{g}_J)_\alpha$ for real roots $\alpha \in \Delta_J$ of $E_{10}$. Tits' relations (R1)-(R7) are satisfied for these generators using the index set $I=\{1,2,\dots, 10\}$.

 The subalgebra $\mathfrak{u}^-$ is the free Lie algebra $L(V')$, where $V' = \bigoplus_{i \in I \setminus J} U(\mathfrak{n}_J^-) \cdot f_i$.
   Thus, $V' = \bigoplus_{k \in \mathbb{N}} U(\mathfrak{n}_{E_{10}}^-) \cdot f_{\lambda_k}$, where $f_{\lambda_k}$ are the Chevalley generators for the imaginary simple roots $\lambda_k$, and $\mathfrak{n}_{E_{10}}^-$ is the negative nilpotent part of $E_{10}$. Each $U(\mathfrak{n}_{E_{10}}^-) \cdot f_{\lambda_k}$ forms an integrable highest weight module for $E_{10}$. The basis elements $S^\prime$ for $V'$ are products of elements from $\mathfrak{n}_{E_{10}}^-$ acting on the highest weight vectors $f_{\lambda_k}$.

As in Subsection~\ref{MonsterEx}, we   construct $G = G(S^\prime) \rtimes G_J$
where $G(S^\prime) = \exp(\widehat{L}(S^\prime))$ and $\widehat{L}(S^\prime)$ is the completion of the free Lie algebra $\mathfrak u^- = L(S^\prime)$.

   Let $x_\alpha \in (\mathfrak{g}_J)_\alpha$ where $\alpha$ is a real root and $x_\alpha $ is a choice of root vector for $\alpha$. Since $V'$ is an integrable $\mathfrak{g}_J$-module (Corollary~\ref{integrable}),   the adjoint action $\text{ad}(x_\alpha)$ is locally nilpotent on $V'$. Hence the series $\exp(\text{ad}(x))(y)$ terminates for any $y \in \widehat{{L}}(S^\prime)$
   
   For elements $h \in \mathfrak{h}_J$, $\text{ad}(h)$ acts diagonally on weight vectors, so $\exp(\text{ad}(h))(y)$ is also well-defined.
   
 Multiplication in $G$ is given by Lemma~\ref{mult}.



\subsection{The gnome Lie algebra}

The gnome Lie algebra,  introduced and studied in \cite{BGN98} and denoted by $\mathfrak{g}_{\mathrm{II}_{1,1}}$, is a Borcherds  algebra formed by taking the Lie algebra of the physical space of the conformal vertex algebra associated with the even unimodular Lorentzian lattice
\[
\mathrm{II}_{1,1} = \left\{ ( m, n )\in \mathbb{Z}\times\Z \mid( m, n) \cdot( m', n') = -m n' - m' n \right\}
\]
denoted in Minkowski coordinates. This lattice can also be realized in a light-cone basis, where elements are represented as pairs $\langle\ell, n\rangle \in \mathbb{Z} \oplus \mathbb{Z}$ with inner product matrix 
$\begin{pmatrix} ~0 & -1\\-1 & ~0\end{pmatrix}$
so that $\langle \ell,n\rangle^2 = -2\ell n$.

 The Cartan subalgebra of $\mathfrak g_{{\rm II}_{1,1}}$ is \( \mathfrak{h} \cong \mathbb{R}^{1,1} \) and 
 $\mathfrak g_{{\rm II}_{1,1}}$ admits a root space decomposition into a direct sum of finite-dimensional subspaces. 
 The root system contains  one real root \( \alpha_{-1} =\langle 1, -1\rangle \), which has an associated $\mathfrak{sl}_2$-triple. 
All other simple roots are imaginary. The Borcherds Cartan matrix of the Borcherds algebra formed by including the `missing' modules of \cite{BGN98} has the same entries as the Borcherds Cartan matrix of the monster Lie algebra, with different root multiplicities.

We define, in the light cone basis:
\[
\delta :=\langle 0, 1\rangle.
\]

Then \( \delta \cdot \delta = 0 \) and  \( \alpha_{-1} \cdot \delta = -1 \), where \( \alpha_{-1} =\langle 1, -1 \rangle \) is the  real simple root.
The element \( \delta \) is in the root lattice but is not a root. However, it plays a role of a Weyl vector and plays a role in the  grading of  \( \mathfrak{g}_{\mathrm{II}_{1,1}} \). Let ${\rm II}_{\mathrm{im}}$ denote the imaginary simple roots. 

The following lemma appears in \cite{BGN98} in the summary of the properties of the root lattice:
\begin{lemma}
    A set of simple imaginary roots for the gnome Lie algebra $\mathfrak{g}_{\mathrm{II}_{1,1}}$ is given by  ${\rm II}_{\mathrm{im}}=\{ \langle \ell ,n  \rangle \mid n \geq \ell \geq 1\}$.
    
\end{lemma}

The Weyl group is \( \mathbb{Z}/2\Z \), generated by reflection in \( \alpha_{-1} \).
As above, every imaginary simple root \( \alpha \) in the fundamental Weyl chamber satisfies:
    \[
    \alpha = m \alpha_{-1} + n \delta =\langle m, n - m\rangle, \quad \text{with }  n > m>0.
    \]
Each such root \( \alpha \) defines a highest weight vector \( f_\alpha \in \mathfrak{g}_{\mathrm{II}_{1,1}} \)
and each root space \( \mathfrak{g}_\alpha \) for a root \( \alpha =\langle m, n - m\rangle \in \Delta \) has dimension
\[
\dim \mathfrak{g}_\alpha = {\rm II}_1(1 + m n),
\]
where \( {\rm II}_1(n) = p_1(n) - p_1(n-1) \), and \( p_1(n) \) is the number of partitions of \( n \) into positive integers.
 
 Free subalgebras of $\mathfrak{g}_{{\rm II}_{1,1}}$ were identified in \cite{BGN98}: the imaginary simple roots are not mutually orthogonal, that is \( \alpha \cdot \alpha' < 0 \) for distinct \( \alpha, \alpha' \in {\rm II}_{\mathrm{im}} \), and hence, by Theorem 5, of \cite{BGN98} and Theorem \ref{JurThm} the subalgebras

\begin{equation}
\mathfrak{u}_{-} =  \coprod_{\alpha} m_\alpha U (\mathfrak g_J)f_{-\alpha}, \quad \alpha \in {\rm II}_{\mathrm{im}}  , \label{eq:gnomeFreeSub}
\end{equation}

\[
\mathfrak{u}_+ =  \coprod_{\alpha} m_\alpha U (\mathfrak g_J)e_\alpha, \quad \alpha \in {\rm II}_{\mathrm{im}}  
\]

are free Lie algebras generated by irreducible highest (respectively lowest) weight vectors given above with muliplicities $m_{\alpha}$.



Similarly, the subalgebra \( \mathcal{H}_- \subset \mathfrak{g}_{\mathrm{II}_{1,1}} \), generated by highest weight vectors \( f_\alpha = \theta(e_\alpha) \), is also a free Lie algebra:
\[
\mathfrak{u}_- = \mathrm{L}\left( \{ f_\alpha \mid r \in {\rm II}_{\mathrm{im}} \} \right).
\]

We have
$$\mathfrak g_J=\mathfrak{sl}_2+\mathfrak h_{{\rm II}_{1,1}}\cong \mathfrak{gl}_2.$$
The vectors $\langle 1,-1\rangle$ and 
$\langle 1,-1\rangle +2\delta$ span $\mathfrak h_{{\rm II}_{1,1}}$. By Theorem~\ref{GroupConstr} we can construct a group associated to $\mathfrak g$. The set $J$ is the index set corresponding to the real root $\alpha_{-1}$. The Cartan subalgebra of $\mathfrak{g}_J$ is $\mathfrak{h}_J \subset \mathfrak{h}$.   The set $I \setminus J$ is set of indices corresponds to the imaginary simple roots of $\mathfrak g_{{\rm II}_{1,1}}$. These are the vectors $\langle \ell, n\rangle$ in light-cone coordinates such that $n \geq \ell \geq 1$ and $\ell \alpha_{-1} + (n-\ell)\delta\in II_{1,1}$ 
Let $\alpha$ be an imaginary simple root with (negative) corresponding Chevalley generator denoted by~$f_{-\alpha}$.
      
   The Lie algebra $\mathfrak g_{{\rm II}_{1,1}}$ admits a decomposition:
    $$ \mathfrak g_{{\rm II}_{1,1}} = \mathfrak{u}^+ \oplus (\mathfrak{g}_J + \mathfrak{h}) \oplus \mathfrak{u}^- $$
    where $\mathfrak{g}_J = \mathfrak{sl}_2$. 
The structure of $\mathfrak{u}^-$ is given by the formula $\mathfrak{u}_{-} = \coprod_{\alpha \in {\rm II}_{\mathrm{im}}} m_\alpha U (\mathfrak g_J)f_{-\alpha}$, where $m_\alpha$ is the multiplicity of the imaginary simple root $\alpha$.
      This multiplicity is denoted by $\mu_{\ell,n}$ in \cite{BGN98}, page 23 and is given by the coefficient of the term $x^n y^{\ell-1}$ in the expansion of the Weyl--Kac--Borcherds denominator formula (Equation 3.5 in \cite{BGN98}).

The module     $U(\mathfrak g_J)\cdot f_{-\alpha}$ represents the highest weight module for $\mathfrak{g}_J = \mathfrak{sl}_2$ generated by the Chevalley generator $f_{-\alpha}$. Since $f_{-\alpha}$ is a highest weight vector with respect to $\mathfrak{sl}_2$ (that is,  $e_{-1} \cdot f_{-\alpha} = 0$), this module is spanned by elements of the form $(f_{-1})^t \cdot f_{-\alpha}$ for $t \ge 0$.
    The direct sum $\coprod_{\alpha \in {\rm II}_{\mathrm{im}}} m_\alpha U (\mathfrak g_J)f_{-\alpha}$ accounts for the multiple copies of these modules when $m_\alpha > 1$.
     The subalgebra  $\mathfrak{u}^-$ is the free Lie algebra $L(V')$, where $V'$ is the vector space defined by this direct sum.

   As before, the tensor algebra $W=T(V')$ is a module for the group $G_J\cong SL_2(\mathbb{K})$. The group $G_J$ acts on $W=T(V')$ via the representation $\pi: \mathfrak{g}_J \to \mathfrak{gl}(W)$.

    Take $S^\prime$ to be the basis for $V'$  formed by taking a basis for each copy of the $\mathfrak{sl}_2$-module $U(\mathfrak g_J)f_{-\alpha}$. Specifically, for each $\alpha=(\ell,n) \in {\rm II}_{\mathrm{im}}$ and each of its $m_\alpha$ copies, we include elements of the form $(f_{-1})^t\cdot f_{-\alpha, k}$ (where $t \ge 0$ and $k=1, \dots, m_\alpha$ indexes the copy).
    The group  $G(S^\prime)$ is the Magnus group $G(S^\prime) = \exp(\widehat{L}(S^\prime))$.
The  semi-direct product group $G$ is:
    $$ G = G(S^\prime) \rtimes G_J. $$
     Multiplication in $G$ is given by Lemma~\ref{mult}.

\bibliographystyle{amsalpha}
\bibliography{MagnusGroup}{}

\end{document}